\documentclass{article}
\usepackage[a4paper,margin=2.5cm]{geometry}

\usepackage{graphicx} 
\usepackage{amsthm}
\usepackage{amsmath,amssymb,amsfonts}
\usepackage{nicematrix}
\usepackage[export]{adjustbox}
\usepackage{float}
\usepackage{algpseudocode}
\usepackage{relsize}
\usepackage[toc,page]{appendix}
\usepackage{algpseudocodex}
\usepackage{algorithm}
\usepackage{comment}
\usepackage{nicefrac}
\usepackage{algorithmicx}
\usepackage{stmaryrd}
\usepackage{caption}
\usepackage[normalem]{ulem}
\usepackage{xspace}
\usepackage{hyperref}

\newtheorem{thm}{Theorem}

\newtheorem{prop}[thm]{Proposition}
\newtheorem{lemma}[thm]{Lemma}

\newcommand{\Tr}{\mathrm{tr}}

\newcommand{\E}{\mathbb{E}}
\newcommand{\lam}{\lambda}

\newcommand{\R}{\mathbb{R}}

\renewcommand{\P}{\mathbb{P}}
\newcommand{\PP}{\mathcal{P}}
\newcommand{\EE}{\mathbb{E}}
\newcommand{\ex}{\mathbb{E}}
\renewcommand{\H}{\mathcal{H}}
\newcommand{\A}{\mathcal{A}}

 \NiceMatrixOptions{code-for-first-row =\scriptstyle,code-for-first-col=\scriptstyle}
\newcommand{\imag}{\text{Imag}}

\newcommand{\tr}{\text{tr}}
\newcommand{\eps}{\varepsilon}
\newcommand{\bigO}{\mathcal{O}}
\algdef{SE}[DOWHILE]{Do}{doWhile}{\algorithmicdo}[1]{\algorithmicwhile\ #1}%
\newcommand{\chebfun}{\texttt{chebfun}\xspace}
\newcommand{\ii}{\mathsf{i}}
\newcommand{\num}{\mathsf{num}}

\newcommand{\mb}[1]{\left[\begin{array}{#1}}
\newcommand{\me}{\end{array}\right]}
\newcommand{\smb}{\left[\begin{smallmatrix}}
\newcommand{\sme}{\end{smallmatrix}\right]}

\usepackage[english]{babel}
\usepackage[fixlanguage]{babelbib}

\title{Stochastic trace estimation for positive trace-class operators}

\author{
Zvonimir Bujanovi\'c\thanks{University of Zagreb, Faculty of Science, Department of Mathematics, Croatia. \texttt{zbujanov@math.hr}}
\and
Luka Grubi\v{s}i\'c\thanks{University of Zagreb, Faculty of Science, Department of Mathematics, Croatia. \texttt{luka.grubisic@math.hr}}
\and
Daniel Kressner\thanks{Institute of Mathematics, EPFL, Switzerland. \texttt{daniel.kressner@epfl.ch}}
\and
Hrvoje Oli\'c\thanks{University of Zagreb, Faculty of Science, Department of Mathematics, Croatia. \texttt{hrvoje.olic@math.hr}}
}

\begin{document}
\maketitle

\begin{abstract}
    Implicit trace estimation aims to approximate the trace of a matrix or linear operator accessible only through matrix-vector or operator-vector products. In the matrix setting, the Girard--Hutchinson estimator typically requires $\mathcal{O}(\varepsilon^{-2})$ products to achieve accuracy $\varepsilon$, while the variance-reduced Hutch++ algorithm reduces this sample complexity
    to $\mathcal{O}(\varepsilon^{-1})$ for positive semidefinite matrices. We develop infinite-dimensional analogues of these estimators for positive trace-class operators on separable Hilbert spaces. The idealized estimators use Gaussian random elements whose covariance is determined by the target operator, leading to unbiased operator versions of Girard--Hutchinson and Hutch++. We prove high-probability error bounds analogous to the finite-dimensional matrix results; in particular, idealized infHutch++ achieves $\mathcal{O}(\varepsilon^{-1})$ sample complexity.
    For practical computation, we introduce truncated implementations that restrict
the random samples to finite-dimensional subspaces; for fixed sample budget, we
show that truncated infHutch++ converges in distribution to its idealized
counterpart as the truncation dimension tends to infinity.
Numerical experiments with integral operators, density-of-states approximations, and spectral filtering for a radial Dirac operator show that these truncated estimators can achieve accuracy comparable to the ContHutch++ algorithm by Zvonek, Horning \& Townsend while using lower-degree function representations and smaller internal discretizations in \chebfun.
\end{abstract}

\section{Introduction}
\label{sec:intro}

Estimating the trace $\tr(A) = a_{11} + \cdots + a_{nn}$ of a symmetric $n \times n$ matrix $A$ that is accessible only through matrix-vector products $x \mapsto Ax$ is a common task in computational mathematics. Known as \emph{implicit} or \emph{matrix-free} trace estimation, this problem arises in many large-scale applications, including methods for approximating spectral densities (and counting eigenvalues), log-determinants, and network measures; see~\cite{cortinovis2021randomizedtraceestimatesindefinite, lin2017randomized, lin2016approximating,  meyer2021hutchoptimalstochastictrace, ubaru2017fast, weisse2006kernel} and the references therein. In principle, the trace could simply be obtained from computing the $n$ matrix-vector products $A e_i$ for all unit vectors $e_1,\ldots,e_n$, followed by forming and summing up $a_{ii} = e_i^T A e_i$. The Girard--Hutchinson trace  estimator~\cite{girard89,hutchinson1989} aims to estimate $\tr(A)$ using far fewer matrix-vector products. It uses $m \ll n$ i.i.d. isotropic random vectors $z_1,\ldots, z_m$ to form 
$$
    \tr(A) \approx H_m(A) = \frac{1}{m}\sum_{i=1}^m z_i^T A z_i.
$$
Gaussian and Rademacher random vectors are the most common choices. 
Their stochastic error is well understood through nonasymptotic concentration bounds~\cite{cortinovis2021randomizedtraceestimatesindefinite,Roosta_Khorasani_2014}, yielding
the characteristic Monte Carlo rate $\bigO(m^{-1/2})$  or $\bigO(\varepsilon^{-2})$
samples for accuracy $\varepsilon$ at fixed failure probability.
Modifications that improve the sample complexity of this basic estimator have been developed~\cite{Epperly_2024, meyer2021hutchoptimalstochastictrace, Persson2022}. In particular, the Hutch++ algorithm \cite{meyer2021hutchoptimalstochastictrace} combines the power of randomized low-rank approximation~\cite{HalkoMartinssonTropp11} with stochastic trace estimation.
For positive semidefinite matrices, Hutch++ achieves relative error $\eps$ with fixed failure probability using $\mathcal O(\eps^{-1})$ matrix-vector products. 

\paragraph{Contributions.}
This work is concerned with the infinite-dimensional analogue: Stochastic trace estimation for a positive trace-class operator $\A$ defined on a separable Hilbert space $\H$. In Section~\ref{sec:optrace}, we first develop an \emph{idealized} Girard--Hutchinson-type estimator for $\A$ that shares the desirable properties of the estimator for the matrix case. In particular, it does not introduce any bias; its expected value is $\tr(\A)$. The combination with the randomized SVD for operators from~\cite{kressner2025randomizedsvdinfinitedimensions} then yields an \emph{idealized} analogue of Hutch++. Theorem~\ref{thm:main} shows this operator variant requires $\bigO(\varepsilon^{-1})$ samples,
matching the $\varepsilon$-dependence of matrix Hutch++.
However, these idealized algorithms are not practical, since one usually
cannot sample directly from the operator-dependent Gaussian distributions they
require. We therefore introduce basis-truncated implementations that restrict
the random inputs to finite-dimensional subspaces of $\H$. For infHutch, the
resulting estimator is exactly Girard--Hutchinson applied to a finite-rank
compression of $\A$. For infHutch++, we show that, for fixed sample budget, the
truncated estimator converges in distribution to its idealized counterpart as
the basis dimension tends to infinity.

\paragraph{Related work.}
The generalization of Gaussian distributions to Hilbert spaces is a well-studied topic in probability theory \cite{da2006introduction}.
Unlike in finite dimensions, however, there is
no $\H$-valued standard Gaussian random element with covariance operator $I$
when $\H$ is infinite-dimensional; the covariance operator of an $\H$-valued
Gaussian random element must be trace class.
Closely related to stochastic trace estimation, Hanson--Wright-type inequalities provide deviation bounds for quadratic forms involving such random elements; see~\cite{mollenhauer2023concentrationsubgaussianvectorspositive} and the references therein. In particular, Proposition 2.2 in~\cite{mollenhauer2023concentrationsubgaussianvectorspositive} establishes a sub-Gamma property for such quadratic forms, generalizing corresponding finite-dimensional results  (see, e.g.,~\cite[Section 2.4]{boucheron2003concentration} and \cite[Lemma 4]{cortinovis2021randomizedtraceestimatesindefinite}). This result also applies in our setting; see Lemma \ref{lema:y-sum-gamma} below.

The main motivation for this work is the ContHutch++ algorithm introduced by Zvonek et al.~\cite{zvonek2024conthutchstochastictraceestimation}. This
algorithm generalizes matrix Hutch++ to positive trace-class integral operators
in a way that differs from the approach pursued here. 
For an integral operator $\A$ on a domain $\Omega$, Zvonek et al. use Gaussian random functions with a
prescribed squared-exponential covariance kernel
$$
    K_{\mathrm{SE}}(x, y) = \frac{1}{ (2\pi \ell^2)^{d/2}} \exp\left(-\frac{\|x-y\|_2^2}{2\ell^2}\right), \quad x, y \in \Omega \subset \R^d,
$$
where $\ell>0$ is a length-scale parameter. 
Drawing independent samples $g_i \sim N(0, K_{\mathrm{SE}})$, the corresponding Girard--Hutchinson-type estimator from~\cite{zvonek2024conthutchstochastictraceestimation} is given by
$$
    H_m^{K_{\mathrm{SE}}}(\A) := \frac{1}{m} \sum_{i=1}^m \langle \A g_i, g_i \rangle,
$$
which we refer to as ContHutch in the following.
Combined this estimator with the infinite-dimensional randomized SVD from~\cite{BoulleTownsend23}, using samples from the same Gaussian process, yields the ContHutch++ algorithm.

An important advantage of ContHutch++ is that the smoothness induced by the
covariance kernel allows the random samples to be represented conveniently, for example in \chebfun~\cite{chebfun}.
On the other hand, the prescribed
covariance introduces a smoothing bias: In general, the expected value of the estimator is not equal to $\tr(\A)$, and the size of the bias depends on $\ell$. The choice of $\ell$ is critical and far from straightforward. Decreasing $\ell$ reduces the smoothing bias of the covariance kernel, but increases the resolution required to represent the random samples. For
example, sampling may require a truncated Karhunen--Lo\`eve expansion or a
factorization of a discretized covariance matrix, whose cost and resolution
requirements increase as $\ell$ decreases.

Our approach avoids the drawbacks of ContHutch++ described above by replacing the externally prescribed Gaussian-process covariance with covariance operators derived from $\A$ itself. Thus, any regularity of the samples is dictated by $\A$ rather than by a separately tuned length-scale parameter.

\paragraph{Structure of the paper.}
In Section~\ref{sec:prelim}, we recall basic facts on trace-class operators and review the Girard--Hutchinson and Hutch++ estimators for matrices. In Section~\ref{sec:optrace}, we develop the corresponding idealized estimators for positive trace-class
operators on separable Hilbert spaces. We first introduce the Gaussian random elements used by infHutch, then combine this construction with
randomized low-rank approximation to obtain infHutch++, 
and analyze basis-truncated implementations, including their consistency as
the basis dimension tends to infinity.
Numerical experiments are presented in Section~\ref{sec:experiments}, comparing the resulting methods with ContHutch and ContHutch++ on integral-operator trace estimation, density-of-states approximation, and quadrature-based spectral filtering for a radial Dirac operator.

\section{Preliminaries}
\label{sec:prelim}

In this section, we recall basic concepts from operator theory needed in our work. To set the stage for stochastic trace estimation in Hilbert spaces, we also briefly recall such methods for matrices.

\subsection{Trace-class operators on a Hilbert space}

Let $\H$ be a separable, infinite-dimensional Hilbert space which, for simplicity, we assume to be real.

Consider a positive trace-class operator $\A : \H \to \H$, that is, $\A$ is bounded, self-adjoint, satisfies $\langle \A x,x \rangle \ge 0$ for all $x \in \H$, and has finite trace
\begin{equation*}
    \Tr(\A) := \sum_{i=1}^\infty \langle \A e_i,e_i\rangle,
\end{equation*}
independent of any orthonormal basis $\{e_1,e_2,...\}$ of $\H$.

A typical example is an integral operator $\A: L^2(\Omega) \to L^2(\Omega)$ on a compact domain $\Omega \subseteq \R$ with a continuous, symmetric, and positive semidefinite kernel $f:\Omega \times \Omega \to \R$:
\begin{equation}
    \label{eq:integral-operator}
    [\A u](x) = \int_\Omega f(x,y)u(y) \, dy.
\end{equation}

Every positive trace-class operator $\A$ has finite Hilbert--Schmidt (HS) norm
$
    \| \A\|_{\mathrm{HS}} := \left(\sum_{i=1}^\infty \| \A e_i\|^2 \right)^{1/2}
$
and admits the spectral decomposition
\begin{equation} \label{eq:spectraldecomp}
    \A = \sum_{i=1}^\infty \sigma_i \langle u_i, \cdot\rangle u_i,
\end{equation}
where $\sigma_1 \geq \sigma_2 \geq \ldots > 0$ are the nonzero eigenvalues and $u_1, u_2, \ldots$ are corresponding orthonormal eigenvectors.
The HS, operator and trace norms are given by
$\|\A\|_{\mathrm{HS}} = \left(\sum_{i=1}^\infty \sigma_i^2 \right)^{1/2}$, $\|\A\|=\sigma_1$ and 
$\|\A\|_{\Tr} = \tr(\A) = \sum_{i=1}^\infty \sigma_i$, respectively.
Moreover, by the Schmidt--Mirsky theorem \cite[Theorem 4.4.7]{eubanks}, a best rank-$k$ approximation of $\A$ (in either of these norms) is given by
\begin{equation*}
    \A_k = \sum_{i=1}^k \sigma_i \langle u_i, \cdot \rangle u_i,
\end{equation*}
with the error given by
\begin{equation*}
    \|\A-\A_k\|_{\mathrm{HS}} = \Big( \sum_{i=k+1}^\infty \sigma_i^2 \Big)^{1/2}, \quad \|\A-\A_k\| = \sigma_{k+1}, \quad \|\A - \A_k\|_{\Tr} = \sum_{i=k+1}^\infty \sigma_i.
\end{equation*}

\paragraph{Quasimatrices.}
To simplify notation, we will make use of quasimatrices \cite{townsend2015continuous} over $\H$ --- the infinite-dimensional analogue to ``tall and skinny'' matrices. A quasimatrix $B$ with $k$ columns contains $k$ vectors $x_1,\ldots,x_k \in \H$: 
    \begin{equation*}
    B = \begin{bmatrix}
        x_1|x_2|\dots|x_k
    \end{bmatrix}. 
    \end{equation*}
For an arbitrary operator $\A: \H \to \H$, we denote 
    \begin{equation*}
    \A B = \begin{bmatrix}
        \A x_1 |\A x_2 | ... | \A x_k
    \end{bmatrix}.
    \end{equation*}
    By viewing the quasimatrix $B$ as an operator $B : \R^k \to \H$, the following relations hold:
    \begin{equation*}
    B^* B =  \begin{bmatrix}
        \langle x_1, x_1\rangle & \dots & \langle x_1, x_k \rangle \\
        \vdots & \ddots & \vdots \\
        \langle x_k,x_1 \rangle & \dots & \langle x_k, x_k\rangle
    \end{bmatrix}, \; \; \;
    BB^* = \sum_{j=1}^k \langle x_j, \cdot \rangle  x_j.
    \end{equation*}
    
Given an orthonormal quasimatrix $Q$ (that is, $Q^*Q = I_k$), the projection of $\A$ onto the subspace spanned by the columns of $Q$ is given by
$$
QQ^*\A = \sum_{j=1}^k \langle q_j, \A(\cdot)\rangle q_j.
$$

\subsection{Girard--Hutchinson estimator}

As mentioned in the introduction, the classical 
estimator by 
Girard \cite{girard89} and Hutchinson \cite{hutchinson1989} for the trace of a symmetric $n \times n$ matrix $A$ is given by
\begin{equation} \label{eq:hutchinson}
    H_m(A) = \frac{1}{m} \sum_{i=1}^m z_i^TAz_i,
\end{equation}
where $z_i \in \R^n$ are i.i.d.~random vectors. Assuming that the vectors $z_i$ are isotropic, i.e., $\E[z_i z_i^T] = I_n$, this estimator is unbiased:
\begin{equation*}
    \E[z_i^TAz_i] = \Tr(A).
\end{equation*}
For specific distributions of $z_i$, non-asymptotic tail bounds have been derived in, e.g.,~ \cite{cortinovis2021randomizedtraceestimatesindefinite,hutchinson1989,Roosta_Khorasani_2014}. 
We use the following bound from~\cite[Theorem 1]{cortinovis2021randomizedtraceestimatesindefinite}.
\begin{thm}
    \label{thm:hutch}
    Consider $\eps > 0$,  $0 < \delta < 1$, and independent standard Gaussian random vectors $z_1,...,z_m \in \R^n$. If
    $$
        m \ge \frac{4}{\eps^2}\left( \|A\|_F^2 + \eps \|A\|_2\right)\log\left(\frac{2}{\delta}\right),
    $$ 
    then the estimator~\eqref{eq:hutchinson} satisfies the error bound
    $$
        |H_m(A) - \Tr(A)|\leq \eps
    $$ 
    with probability at least $1 - \delta$.
\end{thm}

\subsection{Hutch++ algorithm}
The statement of Theorem \ref{thm:hutch} is characteristic for Monte Carlo methods: The required sample size increases proportional with $\eps^{-2}$ as the accuracy $\eps$ decreases, rendering such methods impractical for smaller $\eps$. In order to overcome this obstacle, Meyer et al.~\cite{meyer2021hutchoptimalstochastictrace} propose to first approximate $A$ by a low-rank matrix $\tilde{A} = Q Q^T A Q Q^T$ with an orthonormal matrix $Q \in \R^{n \times (m/3)}$.
Cyclicity of the trace gives the decomposition
\[
    \Tr(A) = \Tr(\tilde{A}) + \Tr(A-\tilde{A})
    =
    \Tr(Q^T A Q)
    +
    \Tr\big((I-QQ^T)A(I-QQ^T)\big).
\]
Hutch++ computes the first term exactly and applies the Girard--Hutchinson
estimator to the second term:
\[
    \Tr(A)
    \approx
    \Tr(Q^T A Q)
    +
    H_{m/3}\big((I-QQ^T)A(I-QQ^T)\big).
\]
Using a randomized range finder~\cite{HalkoMartinssonTropp11} for constructing $Q$
results in Algorithm \ref{alg:hutchPP}, which still uses a total of $m$ matrix-vector multiplications. With high probability, the matrix $Q$ captures a good part of the dominant eigenspaces of $A$, which in turn reduces the variance of the Girard--Hutchinson estimator. While this effect is particularly dramatic when $A$ features strong eigenvalue decay, the sample complexity is always reduced to $\bigO(\eps^{-1})$ for fixed failure probability,
as stated in the following result for a symmetric positive semidefinite (SPSD) matrix.

\begin{thm}[{\cite[Theorem 1]{meyer2021hutchoptimalstochastictrace}}]
Consider $\eps > 0$,  $0 < \delta < 1$, and an SPSD matrix $A \in \R^{n \times n}$.
Then with $m = \bigO\big(\sqrt{\log(\nicefrac{1}{\delta})}/\eps + \log{(\nicefrac{1}{\delta})} \big)$ matrix-vector products, the output $\mathrm{H}_{++}(A)$ of Algorithm \ref{alg:hutchPP} satisfies
    \begin{equation*}
        (1-\eps)\Tr(A) \leq \mathrm{H}_{++}(A) \leq (1+\eps)\Tr(A)
    \end{equation*}
    with probability at least $1-\delta$.
\end{thm}

\begin{algorithm}
    \caption{Hutch++ for matrices \cite[Algorithm 1]{meyer2021hutchoptimalstochastictrace}}
    \label{alg:hutchPP}
    \begin{algorithmic}
        \State \textbf{Input:} SPSD matrix $A \in \R^{n \times n}$ (accessed through matrix-vector products); integer m (divisible by 3)
        \State \textbf{Output:} Approximation of $\Tr(A)$
        \State Sample independent Gaussian random matrices $S \in \R^{n \times (m/3)}$ and $ G \in \R^{n \times (m/3)}$.
        \State Compute an orthonormal basis $Q$ for span of $AS$.
        \State \textbf{Return} $\mathrm{H}_{++}$($A$) = $\Tr(Q^TAQ) + \frac{3}{m}\Tr(G^T(I-QQ^T)A(I-QQ^T)G)$.
    \end{algorithmic}
\end{algorithm}

\section{Operator trace estimation}
\label{sec:optrace}

Algorithm~\ref{alg:hutchPP} samples (standard) Gaussian random vectors. A meaningful infinite-dimensional generalization of this algorithm needs to sample Gaussian random elements from the infinite-dimensional Hilbert space $\H$ instead. The ContHutch++ algorithm presented in~\cite{zvonek2024conthutchstochastictraceestimation} is such a generalization that samples from Gaussian processes with squared exponential kernels. In the spirit of~\cite{kressner2025randomizedsvdinfinitedimensions},
the approach pursued in this work instead uses Gaussian random elements with covariance operators derived from $\A$.  This approach makes the algorithm easier to analyze; in fact, its properties become a direct generalization of the matrix case. Importantly, and unlike the algorithms from~\cite{zvonek2024conthutchstochastictraceestimation}, our (idealized) infinite-dimensional generalizations produce unbiased trace estimators. 
To make the implementation of these algorithms viable, we will also sketch how our algorithms combine with discretizations of $\A$.

\subsection{Unbiased Girard--Hutchinson estimator{: infHutch}}
\label{sec:unbiased-operator-Hutch}

Given a positive trace-class operator $\A : \H \to \H$, our estimators use random elements in $\H$ from the Gaussian distribution with covariance operator $\A$. 
To provide details on this distribution, let us consider the spectral decomposition~\eqref{eq:spectraldecomp} of $\A$ with  nonzero eigenvalues $\sigma_1 \ge \sigma_2 \ge \cdots > 0$, and corresponding orthonormal eigenvectors $u_1,u_2,\ldots$. 
By Kolmogorov's extension theorem \cite[Theorem 2.1.21]{durrett}, we can consider a sequence of i.i.d. standard Gaussian random variables $\{\omega_i\}_{i=1}^\infty$, and define
\begin{equation} \label{eq:gauss_inf}
x^{(k)} = \sum_{i=1}^k \omega_i \sqrt{\sigma_i} u_i.
\end{equation}
Since $\A$ is trace-class, Tonelli's theorem \cite[Proposition 5.2.1]{cohn_measure_theory} yields $\ex[\sum_i \sigma_i \omega_i^2] < \infty$. Together with its nonnegativity, this implies that $\sum_i \sigma_i \omega_i^2$ is finite almost surely. 
In turn, $x^{(k)}$ converges almost surely to a random element $x$ in $\H$. The distribution of $x$ coincides with the Gaussian distribution $N(0, \A)$ with mean $0$ and covariance $\A$ as defined in \cite{da2006introduction}.

We now consider independent random elements $x_1,...,x_m \in \H$ with $x_i \sim N(0, \A)$, and define the random variable
\begin{equation}\label{eq:inf-hutch}
    H_m(\A) = \frac{1}{m}\sum_{i=1}^m \|x_i\|^2,
\end{equation}
as the operator variant of the Girard--Hutchinson estimator. {We refer to this estimator as (idealized) infHutch in the following text.} Monotone convergence implies that {infHutch} is an unbiased estimator for $\Tr(\A)$:
$$
    \E[\|x\|^2] 
        = \E\Big[\Big\|\lim_k\sum_{i=1}^k \omega_i \sqrt{\sigma_i} u_i\Big\|^2\Big] 
        = \lim_k \sum_{i=1}^k \sigma_i\ex\left[ \omega_i^2\right] = \Tr(\A).
$$
To get tail bounds, we will use the sub-gamma property of $\|x\|^2 - \Tr(\A)$. Recall that a centered random variable $X$ is called sub-gamma\footnote{Note that this definition, sometimes referred to as two-sided sub-gamma, requires that the MGF of $X$ is bounded on both tails, that is, both for positive and negative values of $\lam$.} with parameters $(\nu, c)$ if its moment generating function (MGF) can be bounded as 
\begin{equation} \label{eq:subgamma}
     \Phi_X(\lam) := \log\left( \E[\exp(\lam X)]\right) \leq \frac{\nu \lam^2}{2(1-c|\lam|)}, \text{ for all } |\lam| < 1/c.
\end{equation}

\begin{lemma} 
    \label{lema:y-sum-gamma}
    The random variable $\|x\|^2 - \Tr(\A)$ is sub-gamma with parameters $(2\|\A\|_{\mathrm{HS}}^2, 2\|\A\|).$
\end{lemma}
\begin{proof}
For non-negative $\lambda$, the proof in~\cite[Proposition 2.2, Appendix A]{mollenhauer2023concentrationsubgaussianvectorspositive} establishes the property~\eqref{eq:subgamma} for $\|x\|^2 - \Tr(\A)$ with the stated parameters in a more general context. This proof extends seamlessly to negative values of $\lam$. Another way to establish the statement of the lemma is to apply dominated convergence to the finite-dimensional version of the same result~\cite[Lemma 4]{cortinovis2021randomizedtraceestimatesindefinite}.
\end{proof}

Now, a tail bound for the estimator $H_m(\A)$ follows from known properties of sub-gamma random variables.

\begin{thm}
    \label{thm:hw-hutch}
    Given a positive trace-class operator $\A$ and independent $x_1,...,x_m \sim N(0,\A)$, the estimator $H_m(\A) = \frac{1}{m}\sum_{i=1}^m \|x_i\|^2$ satisfies for every $\varepsilon > 0$ the tail bound
    \begin{equation*}
    \P(|H_m(\A) - \Tr(\A)|>\varepsilon) \leq 2 \exp\left( - \frac{m \eps^2}{4(\|\A\|_{\mathrm{HS}}^2+\eps\|\A\|)}\right).
    \end{equation*}
\end{thm}
\begin{proof}
The result follows from the arithmetic with sub-gamma random variables. By
Lemma~\ref{lema:y-sum-gamma} and scaling, 
    $
        Z := H_m(\A) - \Tr(\A) = \frac{1}{m} \sum_{i=1}^m \left( \|x_i\|^2 - \Tr(\A) \right)
    $
is a sum of $m$ independent sub-gamma random variables with parameters $(2\|\A\|_{\mathrm{HS}}^2 / m^2, 2\|\A\| / m)$.
By Proposition 5.1 (a) in~\cite{Zhang2021}, it follows that $Z$ is sub-gamma with parameters 
$(\nu,c) = (2\|\A\|_{\mathrm{HS}}^2 / m, 2\|\A\| / m)$. Thus, the result of the theorem follows from the standard tail bound
for sub-gamma random variables,
     $$
         \P(|Z| > \eps) \leq 2\exp\left(\frac{-\eps^2}{2(\nu+c\eps)} \right), \quad \eps > 0.
     $$    
see, e.g.,~\cite[Proposition 2.10]{Wainwright19}.
\end{proof}

Since $\|\A\|_{\mathrm{HS}},\|\A\|\leq \Tr(\A)$, Theorem~\ref{thm:hw-hutch}
implies that, for fixed failure probability, infHutch requires
$\bigO(\eps^{-2})$ samples to achieve relative accuracy $\eps$.

Let us highlight, once more, that our framework and analysis are a direct generalization of the finite-dimensional setting.

In particular, the bound of Theorem~\ref{thm:hw-hutch} coincides with the corresponding result for the matrix case \cite[Theorem 5]{cortinovis2021randomizedtraceestimatesindefinite}.

\subsection{The infHutch++ algorithm}

Following the idea of the Hutch++ algorithm for matrices~\cite{meyer2021hutchoptimalstochastictrace} and the ContHutch++ algorithm for operators~\cite{zvonek2024conthutchstochastictraceestimation}, we split the trace estimator into two parts.
First, given the orthogonal projector $QQ^*$ onto a suitable finite-dimensional subspace spanned by the columns of an orthonormal quasimatrix $Q$, the operator $\A$ is approximated by the low-rank operator $QQ^* \A QQ^*$. The trace of this approximation coincides with $\Tr(Q^* \A Q)$, yielding the splitting
\begin{equation} \label{eq:hutchppslit}
 \Tr(\A) = \Tr(Q^* \A Q) + \Tr(\Delta), \quad \text{with} \quad \Delta :=(I-QQ^*)\A(I-QQ^*). 
\end{equation}
To determine $Q$, we use the infinite-dimensional randomized SVD from~\cite{kressner2025randomizedsvdinfinitedimensions} and choose a subspace that is spanned by samples from $N(0, \A \A^*) = N(0, \A^2)$. 
To estimate the remainder $\Tr(\Delta)$ in~\eqref{eq:hutchppslit}, we use the Girard--Hutchinson estimator described in Section~\ref{sec:unbiased-operator-Hutch}. Instead of sampling $x_i \sim N(0,\Delta)$ and computing $\|x_i\|^2$, we can equivalently sample 
$g_i \sim N(0,\A)$ and compute $\|(I-QQ^*) g_i\|^2 = \langle g_i, (I-QQ^*) g_i \rangle$.
This follows from
\[
    (I-QQ^*)g_i \sim N\bigl(0,(I-QQ^*)\A(I-QQ^*)\bigr)
    = N(0,\Delta).
\]
In particular, conditioned on $Q$, the resulting
estimator of $\Tr(\Delta)$ is unbiased.
The resulting (idealized) infHutch++ algorithm is summarized in Algorithm~\ref{alg:operator-hutchpp}.
Note that $G$ is understood to be independent of $S$.

\begin{algorithm}
    \caption{infHutch++}
    \label{alg:operator-hutchpp}
    \begin{algorithmic}
        \State \textbf{Input:} Positive trace-class operator $\A$, integer $m$ (divisible by 3) determining number of samples
        \State \textbf{Output:} Approximation of $\Tr(\A)$
        \State Draw $m/3$ independent samples from $N(0,\A^2)$ and collect them in quasimatrix $S$.
        \State Compute an orthonormal basis $Q$ for span of $S$.
        \State Draw $m/3$ independent samples from $N(0,\A)$ and collect them in quasimatrix $G$.
        \State Compute $H_{m/3}( \Delta ) = \frac{3}{m}\Tr(G^* (I-QQ^*) G)$
        \State \textbf{Return} $\hat{T}(\A)$ = $\Tr(Q^* \A Q) + H_{m/3}( \Delta )$.
    \end{algorithmic}
\end{algorithm}

In the following, we will show that infHutch++ reduces the $\bigO(\eps^{-2})$ sample complexity of Girard--Hutchinson to $\bigO(\eps^{-1})$, mirroring the corresponding result~\cite[Theorem 1]{meyer2021hutchoptimalstochastictrace} for Hutch++. For this purpose, we will make use of existing results 
on the (randomized) low-rank approximations of $\A$.

\begin{prop}[{\cite[Theorem 5.1]{zvonek2024conthutchstochastictraceestimation}}]
    \label{prop:rank-k-bound}
    Let $\A_k$ for $k \ge 1$ denote the best rank-$k$ approximation of a positive trace-class operator $\A$. Then
    $$
        \|\A - \A_k\|_{\mathrm{HS}} \leq \frac{1}{\sqrt{k}} \Tr(\A).
    $$
\end{prop}

Instead of the best rank-$k$ approximation $\A_k$, Algorithm \ref{alg:operator-hutchpp} implicitly computes a randomized approximation to $\A_k$. To analyze that approximation, let us first make this approximation more explicit. Given an oversampling parameter $p$, we consider the following procedure.
\begin{enumerate}
    \item Sample a random quasimatrix $S= [s_1, s_2,...,s_{k+p}]$  with $s_i \sim N(0,\A \A^*)$ i.i.d.
    \item Compute an orthonormal basis $Q$ of range$(S)$.
    \item Return finite-rank operator $\Tilde{\A} = QQ^*A$.
\end{enumerate}

\begin{prop}[{\cite[Theorem 2]{kressner2025randomizedsvdinfinitedimensions}}]
    \label{prop:infsvd}
    Consider $\A_k,\Tilde{\A}$ defined as above and suppose that $k \geq 2$ and $p \geq 4$. Then for all $u,t \geq 1$, 
    the inequality 
    $$
     \|\A - \Tilde{\A}\|_{\mathrm{HS}}\leq \|\A - {\A_k}\|_{\mathrm{HS}} + \eta, \quad \text{with } \eta = t  \sqrt{\frac{3k}{p+1}} \|\A - {\A_k}\|_{\mathrm{HS}} + ut \frac{e \sqrt{k+p}}{p+1}\|\A - {\A_k}\|
    $$
    holds with probability at least $1 - 2t^{-p}-e^{-u^2/2}$. 
\end{prop}

To continue from here, we first generalize an auxiliary result from~\cite[Theorem 4]{meyer2021hutchoptimalstochastictrace} to operators.
\begin{lemma}
    \label{lemma:A-is-Ahat-plus-Delta}
    Let $\A$ be a positive trace-class operator, let $0 < \delta < 1$, and consider integers $k\ge 1$, $\ell \ge 1$. Let $\hat{\mathcal{A}}$ and $\Delta$ be any positive trace-class operators satisfying
    $$
        \Tr(\A) = \Tr(\hat{\A}) + \Tr(\Delta), \quad \|\Delta\|_{\mathrm{HS}} \leq 2 \|\A - \A_k \|_{\mathrm{HS}},
    $$
    where $\A_k$ is the best rank-$k$ approximation of $\A$. 
    If $\ell \ge 4 \log(2/\delta)$, then with probability at least $1 - \delta$,
    $$
        |\Tr(\hat{\A}) + H_{\ell}(\Delta)  - \Tr(\A)| \leq 8 \sqrt{\frac{\log(2/\delta)}{k \ell}} \cdot \Tr(\A).
    $$
\end{lemma}
\begin{proof}
    The following proof is an adaptation of the proof of \cite[Theorem 4]{meyer2021hutchoptimalstochastictrace} to our setting. By Theorem~\ref{thm:hw-hutch} applied to the operator $\Delta$, 
    \begin{equation} \label{eq:auxfail}
     \P(|H_\ell(\Delta) - \Tr(\Delta)|>\varepsilon) \leq 2 \exp\left( - \frac{\ell \eps^2}{4(\|\Delta\|_{\mathrm{HS}}^2+\eps\|\Delta\|)}\right).
    \end{equation}
    If $\Delta = 0$, the claim is immediate. Otherwise, 
    choose $\varepsilon = 4 \|\Delta\|_{\mathrm{HS}} \sqrt{ \log(2/\delta) / \ell}$. Noting that $\sqrt{ \log(2/\delta) / \ell} \le 1/2$, we obtain that     $
     \|\Delta\|_{\mathrm{HS}}^2+\eps\|\Delta\| \le 3 \|\Delta\|_{\mathrm{HS}}^2. 
    $
    In turn, the exponent in the right-hand side of~\eqref{eq:auxfail} is at least $4 \log(2/\delta) /3$ and, hence, the failure probability is at most $2 \exp( - 4 \log(2/\delta) /3) \le \delta$. Hence, with probability at least $1-\delta$,
    \[
     |H_\ell(\Delta) - \Tr(\Delta) | \le 4 \sqrt{\frac{\log(2/\delta)}{\ell}} \|\Delta\|_{\mathrm{HS}} 
     \le 8 \sqrt{\frac{\log(2/\delta)}{\ell}} \|\A - \A_k \|_{\mathrm{HS}} \le 8 \sqrt{\frac{\log(2/\delta)}{k\ell}}
     \Tr(\A),
    \]
    where we used Proposition~\ref{prop:rank-k-bound} for the last inequality. Using that 
    $\Tr(\Delta) = \Tr(\A) - \Tr(\hat{\A})$ completes the proof.
\end{proof}

Finally, the following theorem establishes $\bigO(\varepsilon^{-1})$ sample complexity of infHutch++.
\begin{thm}
\label{thm:main}
There is a universal constant $C$ such that the following holds.
Let $\A$ be a positive trace-class operator and let
$0<\eps<1$, $0<\delta<1$. If infHutch++
(Algorithm~\ref{alg:operator-hutchpp}) is executed with an integer $m$
divisible by $3$ and satisfying
\[
    m \ge
    C\big(
        \sqrt{\log(2/\delta)}\,\eps^{-1}
        + \log(2/\delta)
    \big),
\]
then its output $\hat{T}(\A)$ satisfies
\begin{equation}
\label{eq:relaux}
    (1-\eps)\Tr(\A)
    \leq \hat{T}(\A)
    \leq (1+\eps)\Tr(\A)
\end{equation}
with probability at least $1-\delta$.
\end{thm}

\begin{proof}
Set
$L=\log(8/\delta)$, $k=\big\lceil 8\sqrt{L}\,\eps^{-1}\big\rceil$,
$p=\frac{m}{3}-k$, and choose $u=\sqrt{2L}$ and $t=2$.
By choosing the universal constant $C$ in the statement sufficiently large,
we can ensure that
$
    p\geq C_0(k+L)
$
for a sufficiently large universal constant $C_0$. In particular,
$p\geq4$ and $m/3=k+p$.

With $\Delta = (I - QQ^*) \A (I - QQ^*)$, 
Proposition~\ref{prop:infsvd} gives
    $$
    \|\Delta\|_{\mathrm{HS}} \leq \|(I-QQ^*)\A\|_{\mathrm{HS}}\leq \|\A - \A_k \|_{\mathrm{HS}} \left( 1 + t\sqrt{\frac{3k}{p+1}} + ut \frac{e \sqrt{k+p}}{p+1} \right)
    $$
    with probability at least $1 - 2t^{-p} - e^{-u^2/2} = 1 - 2^{1-p} - e^{-L} \ge 1- \delta/2$.
    The same choice of $C_0$ ensures
$
    1+t\sqrt{\frac{3k}{p+1}}
    +ut\,e\,\frac{\sqrt{k+p}}{p+1}
    \leq 2,
$
and hence
$
    \|\Delta\|_{\mathrm{HS}}
    \leq 2\|\A-\A_k\|_{\mathrm{HS}}$.

Set $\hat{\A}=QQ^*\A QQ^*$. Conditioned on $Q$, the residual samples
are independent of $Q$ and yield the estimator $H_{m/3}(\Delta)$.
Moreover, the choice of $C_0$ ensures
$
    \frac{m}{3}=k+p\geq4\log(4/\delta).
$
We may therefore apply Lemma~\ref{lemma:A-is-Ahat-plus-Delta},
conditioned on $Q$, with $\ell=m/3$ and failure probability $\delta/2$.
It follows that
\[
    |\hat{T}(\A)-\Tr(\A)|
    \leq
    8\sqrt{\frac{\log(4/\delta)}{k(k+p)}}\,\Tr(\A)
    \leq
    8\frac{\sqrt{L}}{k}\,\Tr(\A)
    \leq
    \eps\Tr(\A)
\]
with conditional probability at least $1-\delta/2$.
Combining this with the range-finder event by a union bound proves
\eqref{eq:relaux} with probability at least $1-\delta$.
\end{proof}

\subsection{Implementation}
\label{sec:implementation}

The infHutch and infHutch++ estimators introduced above are idealized, since they require sampling from operator-dependent Gaussian distributions. For practical computation, we need to restrict the random inputs to finite-dimensional subspaces of $\H$.
\subsubsection{Truncated infHutch}

To give a concrete example of how the restriction to finite dimension can be achieved, let us consider an orthonormal basis $\{q_1,q_2,\ldots\}$ of $\H$ and consider the finite-dimensional subspace spanned by the first $n$ vectors $q_1,q_2,\ldots,q_n$.
Choose the random vector model $\hat{x} = \gamma_1 q_1 + \cdots + \gamma_n q_n$ with i.i.d. $\gamma_j \sim N(0,1)$ and consider
\begin{equation} \label{eq:approxhutch}
    \hat{H}_m(\A) := \frac{1}{m} \sum_{i=1}^m \langle \A \hat{x}_i, \hat{x}_i \rangle, 
        \quad \hat{x}_i \sim \hat{x} \text{ i.i.d.}
\end{equation}
Restricting the samples to a finite-dimensional subspaces introduces a truncation \emph{bias}, which we quantify below. Two different views provide further insight into~\eqref{eq:approxhutch}.

On the one hand, we can connect~\eqref{eq:approxhutch} to the classical matrix Girard--Hutchinson estimator. To see this, let $Q_n$ denote the quasimatrix containing the columns $q_1,\ldots,q_n$, and form the discretized matrix $A_n = Q_n^*\A Q_n \in \R^{n \times n}$. Then the estimator $\hat{H}_m(\A)$ coincides with Girard--Hutchinson applied to $A_n$:
\begin{equation}
    \label{eq:hatHm}
    \hat{H}_m(\A) 
        = \frac{1}{m} \sum_{i=1}^m \langle \A Q_n \gamma^{(i)}, Q_n \gamma^{(i)} \rangle
        = \frac{1}{m} \sum_{i=1}^m \langle A_n \gamma^{(i)}, \gamma^{(i)} \rangle = H_m(A_n),
\end{equation}
with i.i.d. $\gamma^{(i)} \sim N(0, I_n)$. 

On the other hand, we can view~\eqref{eq:approxhutch} as an unbiased trace
estimator 
of the compressed operator $\A_n = \PP_n \A \PP_n: \H \to \H$ with $\PP_n = Q_n Q_n^*$ denoting the orthogonal projector onto $\operatorname{span}\{q_1,\ldots,q_n\}$.
The restriction of $\A_n$ to $\operatorname{range}(\PP_n)$ is unitarily
equivalent to the matrix $A_n=Q_n^*\A Q_n$, and hence
$\Tr(\A_n)=\Tr(A_n)$. It follows from the construction of infHutch in
Section~\ref{sec:unbiased-operator-Hutch} and the orthogonal invariance of
Gaussian random vectors that
\begin{equation}
\label{eq:mainrel}
    H_m(\A_n)\stackrel{d}{=}H_m(A_n)
    =\hat H_m(\A).
\end{equation}
For this reason, we refer to~\eqref{eq:approxhutch} as truncated infHutch
using the basis $\{q_1,q_2,\ldots\}$.
Note that, although this procedure can be interpreted as approximating the trace of the matrix $A_n$, the matrix $A_n$ itself is never formed or stored. Instead, the implementation samples $\hat{x}\in\H$, computes $\A\hat{x}$, and evaluates the quadratic form $\langle \A\hat{x},\hat{x}\rangle$.

One consequence of~\eqref{eq:mainrel} is the error bound
\begin{align}        
    |\Tr(\A) - \hat{H}_m(\A)|
        &= |\Tr(\A) - H_m(A_n)| \nonumber \\
        &\leq \Tr(\A) - \Tr(\A_n) + |H_m(A_n) - \Tr(A_n)| \nonumber \\
        &= \Tr( (I- \PP_n)\A (I- \PP_n)) + |H_m(A_n) - \Tr(A_n)|. \label{eq:discretization-triangle-ineq}
\end{align}
The first term of the error bound~\eqref{eq:discretization-triangle-ineq} is the discretization error, which can be made arbitrarily small by choosing $n$ sufficiently large. For example, if $\A$ is an integral operator with a smooth kernel on an interval, choosing an $L^2$-orthonormal basis $q_1, q_2, \ldots$ of (scaled) Legendre polynomials will result in a rapid decay
of $\Tr( (I- \PP_n)\A (I- \PP_n))$ to zero 
as $n$ increases; see~\cite{kressner2025randomizedsvdinfinitedimensions} for a related discussion.
Using $\|A_n\|_{F} \le \|\A\|_{\mathrm{HS}}$
and $\|A_n\|_2 \le \|\A\|$, it follows that the second term of the error bound~\eqref{eq:discretization-triangle-ineq} is governed by the tail bound established in Theorem~\ref{thm:hw-hutch}.

\subsubsection{Truncated infHutch++}
For the Hutch++ variant, the idealized Algorithm~\ref{alg:operator-hutchpp} can be implemented using the same finite-dimensional sampling principle. There are, however, two natural ways to treat the range-finder samples.

First, consider a projected variant. After drawing
\begin{equation}
\label{eq:sample-infhutchpp}
    \hat x=Q_n\gamma\in\operatorname{range}(\PP_n) = \operatorname{span}\{q_1,\ldots,q_n\},
    \qquad
    \gamma\sim N(0,I_n),
\end{equation}
one could use $\PP_n\A\hat x$ as the range-finder sample. In coordinates
with respect to $q_1,\ldots,q_n$, this is simply
$
    Q_n^*\PP_n\A Q_n\gamma=A_n\gamma.
$
If the residual samples are restricted to $\operatorname{range}(\PP_n)$
as well, the resulting method is exactly matrix Hutch++ applied to $A_n$,
or equivalently idealized infHutch++ applied to
$\A_n=\PP_n\A\PP_n$. The additional deterministic bias caused by this
compression is
\[
    \Tr(\A)-\Tr(\A_n)
    =
    \Tr\big((I-\PP_n)\A(I-\PP_n)\big).
\]

However, we do not use this projected variant in the numerical implementation. In a matrix-free operator setting, the oracle naturally returns $\A\hat x$ as an element of the ambient Hilbert space. Projecting this output back to $\operatorname{range}(\PP_n)$ would require computing its coefficients in the chosen sampling basis and would discard the component of $\A\hat x$ outside that space. Thus the projected method approximates the compressed operator $\PP_n\A\PP_n$, while our implementation uses the actual images of the sampled input space under $\A$. This choice is more natural in the \chebfun implementation, but it leads to a different finite-$n$ random sketch.

In the implementation used below, sampling from $N(0,\A\A^*)$ for the columns of $S$ is replaced by drawing $\hat{x}$ as in \eqref{eq:sample-infhutchpp}, and storing $\hat s=\A\hat x$ in $S$.
We also use $\hat H_{m/3}$ in place of $H_{m/3}$ in the last line of Algorithm~\ref{alg:operator-hutchpp}, using fresh Gaussian samples independent of those used to construct $Q$.
We refer to the resulting unprojected implementation as ``truncated infHutch++ using the basis $\{q_1,q_2,\ldots\}$''.

For finite \(n\), the range-finder samples in this unprojected variant have covariance $\operatorname{Cov}(\hat s) = \A\PP_n\A$, while the idealized Algorithm~\ref{alg:operator-hutchpp} uses samples from $N(0,\A^2)$. Thus the range finder used by truncated infHutch++ is not exactly the one analyzed in Theorem~\ref{thm:main}. The two covariances nevertheless converge as $n\to\infty$: since
$\PP_n\to I$ strongly and $\A$ is positive trace class (and therefore Hilbert--Schmidt),
$$
    \|\A\PP_n\A-\A^2\|_{\Tr}
        = \|\A(I-\PP_n)\A\|_{\Tr}
        = \|(I-\PP_n)\A\|_{\mathrm{HS}}^2
        \longrightarrow 0.
$$
This covariance convergence suggests consistency of the finite-dimensional
range finder. Below we show that, for a fixed sample budget, the truncated infHutch++ estimator converges to its idealized counterpart as $n\to\infty$. This does not, however, provide a finite-$n$ Hutch++ error bound.

The error decomposition~\eqref{eq:discretization-triangle-ineq} is specific to truncated infHutch and does not directly apply to the above unprojected Hutch++ variant. To describe the residual term, let $Q$ denote the orthonormal quasimatrix computed from the unprojected range-finder samples, and define $\Delta_Q := (I-QQ^*)\A(I-QQ^*)$. Conditioned on $Q$, the residual estimator uses independent samples $\hat{x}_i \sim N(0,\PP_n)$, and hence estimates $\Tr(\PP_n\Delta_Q\PP_n)$:
$$
    \EE\big[
        \hat H_{m/3}(\Delta_Q) \;|\; Q
    \big]
        = \Tr(\PP_n\Delta_Q\PP_n).
$$
Consequently, if $\hat T_n(\A)$ denotes the output of truncated infHutch++, then
\begin{equation}
\label{eq:condbias} 
    \EE \big[\hat T_n(\A)\mid Q\big]
    = \Tr(Q^*\A Q) + \Tr(\PP_n\Delta_Q\PP_n).
\end{equation}
Since $\Tr(\A) = \Tr(Q^*\A Q) + \Tr(\Delta_Q)$, the conditional bias is
$$
    \Tr(\A) - \EE\big[\hat T_n(\A) \;|\; Q\big]
        = \Tr \big((I-\PP_n)\Delta_Q(I-\PP_n)\big)
        \ge 0.
$$
Equivalently,
$$
    |\Tr(\A)-\hat T_n(\A)|
    \le 
        \Tr\big((I-\PP_n)\Delta_Q(I-\PP_n)\big) 
        + \big| \hat H_{m/3}(\Delta_Q) - \Tr(\PP_n\Delta_Q\PP_n) \big|.
$$
Conditioned on $Q$, the second term is a finite-dimensional
Girard--Hutchinson error and can be bounded using
Theorem~\ref{thm:hw-hutch}. For fixed $Q$, the first term tends to zero
as $n\to\infty$. In the actual algorithm, however, $Q=Q_n$ also depends
on $n$. The following result establishes consistency in this setting,
without providing a finite-$n$ counterpart of Theorem~\ref{thm:main}.

\begin{prop}
\label{prop:truncated-infhutchpp-convergence}
Fix an integer $m$ divisible by $3$. Let $\hat T(\A)$  and $\hat T_n(\A)$ denote the outputs
of infHutch++ and truncated infHutch++ defined above. Then
$\hat T_n(\A)$ converges in distribution to $\hat T(\A)$ and
$
    \EE[\hat T_n(\A)]\longrightarrow \Tr(\A).
$
\end{prop}

\begin{proof}
Set $r=m/3$. 
We couple the range-finder samples
by taking i.i.d.\ standard Gaussian variables $\omega_{ij}$ and setting
\[
    s_i^{(n)}
        = \sum_{j=1}^n \omega_{ij}\A q_j,
    \qquad
    s_i
        = \sum_{j=1}^\infty \omega_{ij}\A q_j,
    \qquad i=1,\ldots,r.
\]
Then $s_i^{(n)}$ and $s_i$ have covariances $\A\PP_n\A$ and $\A^2$,
respectively. Since $\A$ is Hilbert--Schmidt,
$\EE\|s_i^{(n)}-s_i\|^2 \to 0$ and $s_i^{(n)}\to s_i$ almost surely.

Let $R_n$ and $R$ denote the orthogonal projectors onto the spans of
$s_1^{(n)},\ldots,s_r^{(n)}$ and $s_1,\ldots,s_r$, respectively.
If $\operatorname{rank}(\A)\ge r$, the limiting samples are linearly
independent almost surely, and continuity of the corresponding range
projector, see~\cite[Lemma~10]{kressner2025randomizedsvdinfinitedimensions},
gives
$\|R_n-R\|\to0$ almost surely.
If $\operatorname{rank}(\A)<r$, both sketches span
$\operatorname{range}(\A)$ almost surely for all sufficiently large $n$,
and the same conclusion follows.

Define
\[
    \Delta_n=(I-R_n)\A(I-R_n),
    \qquad
    \Delta=(I-R)\A(I-R),
    \qquad
    \mathcal B_n=\PP_n\Delta_n\PP_n.
\]
Then
$\|\Delta_n-\Delta\|_{\Tr}
        \le 2\|\A\|_{\Tr}\|R_n-R\|
        \to 0$
almost surely. Since $\Delta$ is trace class and $\PP_n\to I$ strongly,
also
\[
    \|\mathcal B_n-\Delta\|_{\Tr}\to 0
\]
and $\Tr(R_n\A)\to\Tr(R\A)$ almost surely.

For the residual estimators, let $\eta_{ij}\sim N(0,1)$ be i.i.d.,
independently of $\omega_{ij}$. For $i=1,\ldots,r$, define
\[
    Y_{i,n}
    =
    \sum_{j,k=1}^n
    \eta_{i,j}\eta_{i,k}
    \langle \mathcal B_n q_j,q_k\rangle, \qquad 
    Y_{i}
    =
    \lim_{N\to \infty} 
    \sum_{j,k=1}^N
    \eta_{i,j}\eta_{i,k}
    \langle  \Delta q_j,q_k\rangle, 
\]
where the limit is in mean square.
Then $Y_{i,n}$ is exactly one sample used by the truncated residual, while
$Y_i$ has the same distribution as $\|x_i\|^2$ for $x_i\sim N(0,\Delta)$, that is,
as one sample of the idealized residual estimator.

Applying the well-known variance formula for stochastic trace estimators (see, e.g.,~\cite[Lemma~9]{AvronToledo11}), we obtain
\[
     \EE\left[
        |Y_{i,n}-Y_i|^2
        \,\middle|\,R_n,R
    \right]
    =
    2\|\mathcal B_n-\Delta\|_{\mathrm{HS}}^2
    + \big( \Tr(\mathcal B_n-\Delta) \big)^2
    \leq
    3\|\mathcal B_n-\Delta\|_{\Tr}^2.
\]
Hence $Y_{i,n}-Y_i$ converges to zero in probability. Applying this
argument to the finitely many residual samples and combining it with
the convergence of the low-rank trace term shows, under this coupling,
that
$
    \hat T_n(\A)-\hat T(\A)\to 0
$
in probability and, hence, $\hat T_n(\A)$ converges to $\hat T(\A)$ in distribution.

Finally, using~\eqref{eq:condbias},
\[
    \Tr(\A)-\EE[\hat T_n(\A)]
    =
    \EE\!\left[
        \Tr\big((I-\PP_n)\Delta_n(I-\PP_n)\big)
    \right].
\]
The quantity inside the expectation is
$
    \Tr(\Delta_n)-\Tr(\mathcal B_n)
    \to
    \Tr(\Delta)-\Tr(\Delta)=0
$
almost surely. It is nonnegative and bounded by $\Tr(\A)$, so dominated
convergence yields
$
    \EE[\hat T_n(\A)]\longrightarrow\Tr(\A).
$
\end{proof}

\subsubsection{Adaptive choice of sampling dimension}

    The discussion above assumes that the sampling dimension $n$ is fixed before the Gaussian samples used in the trace estimator are drawn. In order to set the value of $n$, we consider an adaptive heuristic algorithm. Note that if $n$ is chosen separately for each sample, then the resulting output would no longer be distributed as $H_m(A_n)$ for a deterministic compression $A_n$. For this reason, we propose to use adaptivity only as a pilot procedure for selecting a basis dimension. After a dimension $n_*$ has been selected by this pilot procedure, the trace estimators are run with this fixed dimension and with fresh independent Gaussian samples. Conditional on the selected value of $n_*$, the fixed-$n$ interpretation from Section~\ref{sec:implementation} therefore applies to the samples used in the reported trace estimates.

    The pilot procedure is based on monitoring the change in quadratic forms. Let
    $$
        \xi_j^{(2n)}
            = \sum_{i=1}^{2n}\omega_{j,i}q_i, 
        \quad
        \delta_j^{(n)}
            = \sum_{i=n+1}^{2n}\omega_{j,i}q_i,
        \quad 
        \omega_{j,i} \sim N(0,1) \; \text{i.i.d.},
    $$
    for $j = 1, \ldots, m_{\text{pilot}}$. These are finite Gaussian expansions in the chosen basis, and the contribution $\langle \A\delta_j^{(n)}, \delta_j^{(n)}\rangle$ has expectation
    $$
        \EE \left[ \langle \A \delta_j^{(n)}, \delta_j^{(n)}\rangle \right]
        = \Tr \big((\PP_{2n}-\PP_n) \A (\PP_{2n}-\PP_n)\big).
    $$
    Thus $\langle \A\delta_j^{(n)}, \delta_j^{(n)}\rangle$ is an unbiased estimator of the trace contribution of the newly added basis block. Note that only finite expansions are used in the pilot procedure; no limiting $\H$-valued standard Gaussian element is required.

    Algorithm~\ref{alg:dynamic-sampling} stops when the empirical contribution of the newest block is small relative to the empirical quadratic form on the enlarged space. In the experiments below, we choose the basis-resolution tolerance $\tau_{\text{basis}}$ as a decreasing function of the sample budget $m$, reflecting that smaller stochastic error makes basis-truncation effects more visible. The criterion is an empirical heuristic: it does not provide a high-probability bound for the full tail $\Tr(\A)-\Tr(\PP_n \A \PP_n)$, and it cannot certify unresolved contributions beyond the last inspected block without additional decay assumptions. The use of several pilot samples is intended to reduce the chance of stopping because of a single atypically small realization. Algorithm \ref{alg:dynamic-sampling} proved reliable in numerical experiments of the next section; we report some further details only for the first experiment.

    \begin{algorithm}[ht]
    \caption{Pilot selection of basis dimension $n$}
    \label{alg:dynamic-sampling}
    \begin{algorithmic}[1]
    \Require operator-vector oracle $\A$; initial dimension $n_0$; maximum dimension
    $n_{\max}$; basis-resolution tolerance $\tau_{\text{basis}}$; number of pilot samples $m_{\text{pilot}}$
    \Ensure selected basis dimension $n_*$
    \State $n \leftarrow n_0$
    \For{$j=1,\ldots,m_{\text{pilot}}$}
        \State Draw independent $\omega_{j,1},\ldots,\omega_{j,n}\sim N(0,1)$
        \State Set $\xi_j^{(n)} \leftarrow \sum_{i=1}^{n}\omega_{j,i} q_i$
    \EndFor
    \vspace*{0.5em}
    \Repeat
        \State $n_{\text{new}} \leftarrow \min\{2n,n_{\max}\}$
        \For{$j=1,\ldots,m_{\text{pilot}}$}
            \State Draw independent \(\omega_{j,n+1},\ldots,\omega_{j,n_\text{new}}\sim N(0,1)\)
            \State \rule{0pt}{1.2em}Set
            $
                \delta_j^{(n)} 
                    \leftarrow \sum_{i=n+1}^{n_\text{new}} \omega_{j,i} q_i,
                \quad
                \xi_j^{(n_{\text{new}})} 
                    \leftarrow \xi_j^{(n)} + \delta_j^{(n)}
            $
        \EndFor
        \If{\rule{0pt}{1.2em}$\frac{1}{m_{\text{pilot}}} \sum_{j=1}^{m_{\text{pilot}}} \langle \A\delta_j^{(n)},\delta_j^{(n)}\rangle \leq \tau_{\text{basis}} \cdot \frac{1}{m_{\text{pilot}}} \sum_{j=1}^{m_{\text{pilot}}} \langle \A\xi_j^{(n_{\text{new}})},\xi_j^{(n_{\text{new}})}\rangle$}
            \State \rule{0pt}{1.2em}\Return $n_* = n_{\text{new}}$
        \Else
            \State $n \leftarrow n_{\text{new}}$
        \EndIf
    \Until{$n \geq n_{\max}$}
    \State \Return $n_* = n_{\max}$
    \end{algorithmic}
    \end{algorithm}

\section{Numerical examples}

\label{sec:experiments}

In this section, we compare the proposed estimators, infHutch and infHutch++, with the corresponding estimators from \cite{zvonek2024conthutchstochastictraceestimation}, denoted ContHutch and ContHutch++. In the numerical experiments, infHutch($n$) and infHutch++($n$) denote the truncated implementations from Section \ref{sec:implementation}, using Legendre polynomials up to and including degree n (i.e.,  using $n+1$ basis functions), while ContHutch($\ell$) and ContHutch++($\ell$) denote the methods from \cite{zvonek2024conthutchstochastictraceestimation} using a squared-exponential covariance kernel with length-scale $\ell$.
All algorithms have been implemented in MATLAB R2022b using the \chebfun package \cite{chebfun}, which greatly simplifies operations with quasimatrices and operators. The code to reproduce the numerical experiments is publicly available at \href{https://github.com/PMF-ZNMZR/infHutch}{https://github.com/PMF-ZNMZR/infHutch}.

To compare the different algorithms, we use the number of operator applications as the primary cost metric. For Girard--Hutchinson this equals the number $m$ of random samples. For Hutch++, $m$ denotes the total operator application budget; $m/3$ random samples are used in each of the range-finder and residual stages. We also report the degree of the Chebyshev polynomials used internally to represent functions and operators in the numerical examples.
Since the computational cost of operations performed by \chebfun scales with the polynomial degree, it is beneficial to keep this degree as small as possible. 

Throughout the numerical experiments, a random sample means one independently generated random function. For the Girard--Hutchinson estimators, $m$ denotes the number of random samples and also the number of quadratic-form evaluations
$\langle \A x_i, x_i \rangle$; in the integral-operator examples this equals $m$ applications of the target operator. For the Hutch++ estimators, $m$ denotes the total number of target-operator applications. These estimators generate $m/3$ random samples for the randomized range finder and $m/3$ random samples for the residual trace estimator, while the remaining $m/3$ operator applications are used to evaluate the deterministic low-rank trace term. When the target operator is a rational function of a differential operator, one application of the target operator requires several shifted linear solves. We therefore use random-sample counts, equivalently target-operator applications, as the primary stochastic trace-estimation metric, and separately report discretization sizes as an indication of the cost of each application.

In the first two numerical experiments, we show the effectiveness of our approach by reconstructing two examples from~\cite{zvonek2024conthutchstochastictraceestimation}. In the third experiment,
we apply trace estimation to a rational spectral filter for an operator with a gap in its essential spectrum. The particular application here is to estimate the number of eigenvalues of an operator in a given interval.

\subsection{Integral kernel operators}

For a given kernel function $f : \Omega \times \Omega \to \R$ on a compact domain $\Omega \subseteq \R$, we are interested in computing the trace of the operator $F:L^2(\Omega) \to L^2(\Omega)$ defined by
$$
[Fu](x) := \int_\Omega f(x,y) u(y) \, dy.
$$
In applications, the kernel $f$ need not be known explicitly; instead, one may only have access to the operator action $u \mapsto Fu$.
For example, $f$ might be the Green's function associated with a differential operator---in that case, the evaluation of the operator $F$ reduces to solving a differential equation.

To demonstrate the convergence of our estimators, as well as to compare their performance to ContHutch, we test them by using two kernel functions given explicitly as in \cite{zvonek2024conthutchstochastictraceestimation}: a Helmholtz-like kernel
\begin{align*}
    f(x,y) 
        &= \frac{1}{1+e^{5(x-y)}}\left[ 1-\cos\left(\frac{\pi(x+1)}{4}\right)\sin\left(\frac{\pi(y+1)}{4}\right)\right] \\ 
        &+ \frac{1}{1+e^{5(y-x)}}\left[ 1 - \cos\left(\frac{\pi(y+1)}{4}\right)\sin\left( \frac{\pi(x+1)}{4}\right)\right],
\end{align*}

and the linear combination of three sinc functions
$$
    f(x,y) = \text{sinc}(x-y) + \frac{1}{2}\text{sinc}(10(x-y)) + \frac{1}{4}\text{sinc}(50(x-y)).
$$
Both kernels are symmetric and positive definite.

In the context of our algorithms, the Hilbert space $\H$ is $L^2([-1, 1])$ with the standard integral inner product. In this particular case, with a known kernel function, we can compute the eigenvalue decomposition of the integral operator $F$ using \chebfun. This allows us to generate random vectors $x^{(k)}$ as in \eqref{eq:gauss_inf} and run idealized infHutch as described in Section \ref{sec:unbiased-operator-Hutch}.

We use idealized infHutch as the benchmark for the truncated infHutch($n$) approximations proposed in Section \ref{sec:implementation}, and take Legendre polynomials over the interval $[-1, 1]$ as the orthonormal basis $\{q_1, q_2, \ldots\}$. This allows for efficient manipulation within \chebfun, as linear combinations of Legendre polynomials are easily converted to linear combinations of Chebyshev polynomials internally used by \chebfun. Accuracy of the algorithm is controlled simply by adjusting the polynomial degree. 
In contrast, ContHutch($\ell$) generates samples from the Karhunen--Lo\`eve (K--L) decomposition of the Gaussian process with squared-exponential covariance kernel $K_{\mathrm{SE}}$.
Increasing the accuracy by decreasing $\ell$ requires recomputing the K--L decomposition, introducing an additional setup cost that is not present when only changing the finite-dimensional basis size $n$. Furthermore, each random sample may have a high polynomial degree when stored using the Chebyshev basis.

\begin{figure}[h]
    \begin{minipage}{0.50\textwidth}
        \includegraphics[width=1\linewidth]{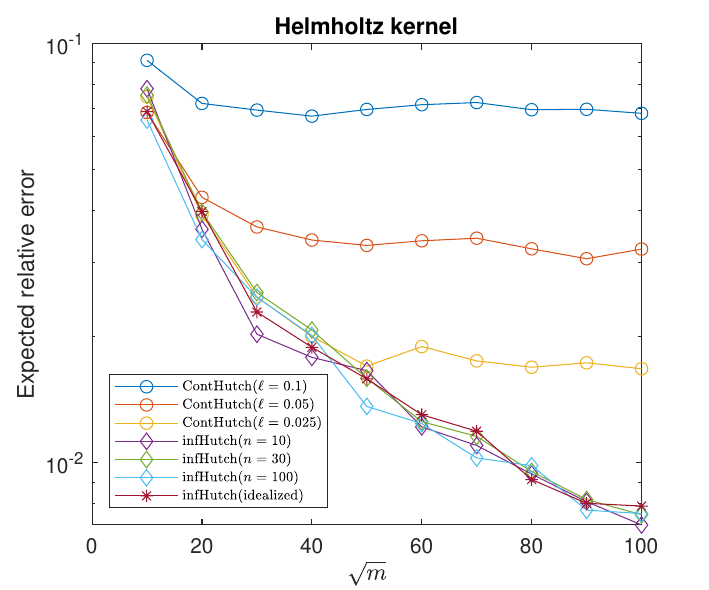}
    \end{minipage}
    \hfill
    \begin{minipage}{0.50\textwidth}
        \includegraphics[width=1\linewidth]{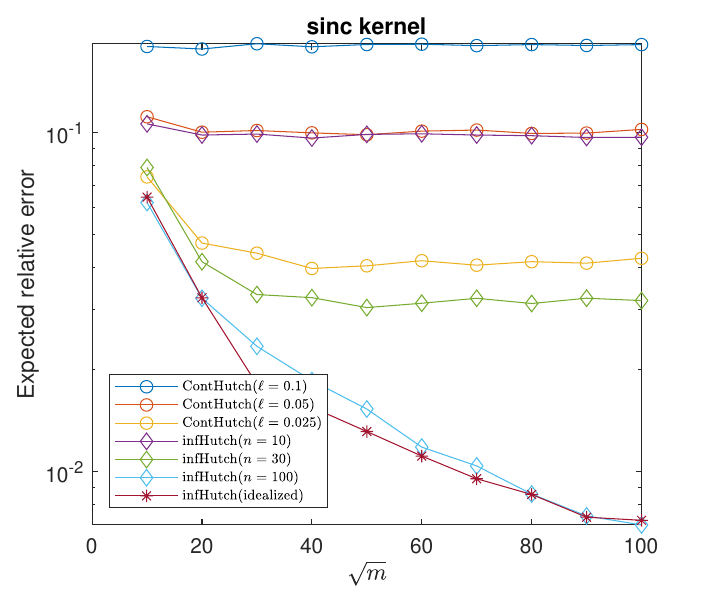}
    \end{minipage}
    \caption{Performance comparison of Girard--Hutchinson's algorithm for different sampling methods on the Helmholtz-like kernel (left panel) and the sinc kernel (right panel). The truncated infHutch($n$) estimators use the Legendre polynomials up to and including degree $n$. On the vertical axis we plot the empirical approximation to $\EE[\frac{|\Tr(F) - \hat T_m|}{\Tr(F)}]$,
    where $\hat T_m$ denotes the trace estimate produced by the plotted method, by averaging the errors over $100$ runs.
    On the horizontal axis we plot $\sqrt{m}$, as the relative error is expected to decay at the Monte Carlo rate $\mathcal{O}(m^{-1/2})$, see Theorem \ref{thm:hw-hutch}.
    }
    \label{fig:hutch-comparison-kernels}
\end{figure}

Figure \ref{fig:hutch-comparison-kernels} shows the behavior of the described sampling approaches for the Girard--Hutchinson estimator. The mathematical idealized infHutch estimator has no tuning parameter. In these experiments, however, we implement it only as a benchmark by truncating the eigenfunction expansion of $F$. The truncation level was chosen sufficiently large that the resulting truncation error is negligible compared with the Monte Carlo errors shown in Figure \ref{fig:hutch-comparison-kernels}. With this implementation, idealized infHutch shows best performance among all algorithms, requiring the least number of random samples in order to reach certain level of relative error in trace estimation. 
The quality of the other estimators is affected by the parameter choice. It may be difficult to select the proper value of the parameter $\ell$ used in the ContHutch($\ell$): if it is too large, the sampled functions will be approximately constant, limiting the achievable accuracy, while if
if $\ell$ is too small, the resolution required to represent and apply the operator to the sampled functions increases dramatically. At first glance, the same issue may appear to be a problem for truncated infHutch($n$) using Legendre polynomials as well, since taking $n$ too small might decrease accuracy, while taking $n$ too large would slow down the computation unnecessarily. To address this issue, as a separate check we used Algorithm~\ref{alg:dynamic-sampling} as a pilot procedure for selecting a suitable finite sampling dimension. We used $m_{\text{pilot}} = 10$ and set the tolerance $\tau_{\text{basis}}$ to $10^{-2}$ for $100$ samples and $10^{-3}$ for $10000$ samples (and interpolated in between: $\tau_{\text{basis}}(m) = 0.1 m^{-1/2}$). Starting from degree $10$, the largest degree selected by Algorithm \ref{alg:dynamic-sampling} over all runs was $20$ for the Helmholtz kernel and $160$ for the sinc kernel. These pilot results are consistent with the fixed-$n$ behavior shown in Figure~\ref{fig:hutch-comparison-kernels}: the fixed-$n$ truncated infHutch estimators using the selected dimensions reach the accuracy of the idealized benchmark. Note that the idealized infHutch can be implemented in this case, but not generally, so we use it only as a benchmark.

Table \ref{table:degree-comparison} compares the degrees of Chebyshev polynomials used internally in \chebfun to represent the random sample functions for selected values of $n$ and $\ell$. Combined with the error curves in Figure \ref{fig:hutch-comparison-kernels}, it shows that parameter choices giving comparable accuracy can require substantially lower Chebyshev degrees for truncated infHutch($n$) than for ContHutch($\ell$), suggesting lower per-sample representation and operator-application cost. Note that increasing the degree of Legendre polynomials over $100$ does not further improve accuracy in trace estimation for either kernel. 
For the benchmark implementation of idealized infHutch, \chebfun represented the resulting random sample functions with Chebyshev polynomials of degree $50$ for the Helmholtz kernel and degree $80$ for the sinc kernel. The samples were generated using the dominant $20$ eigenfunctions in the Helmholtz case and the dominant $40$ eigenfunctions in the sinc case. These truncation levels were chosen so that the omitted eigenvalue tail was negligible relative to the Monte Carlo errors reported here.

\begin{table}[h]
    \centering
    \begin{tabular}{c|c}
        truncated infHutch($n$) & ContHutch($\ell$) \\
        \hline
        $n = 10 \rightarrow \text{deg} = 10$ & $\ell = 0.1 \rightarrow \text{deg} = 94$ \\
         $n = 30 \rightarrow \text{deg} = 30$ & $\ell = 0.05 \rightarrow \text{deg} = 177$ \\
         $n = 100 \rightarrow \text{deg} = 100$ & $\ell = 0.025 \rightarrow \text{deg} = 349$
    \end{tabular}
    \caption{Comparison of degrees of Chebyshev polynomials (denoted as $\mathrm{deg}$) used internally in \chebfun for representing the random samples. 
    }
    \label{table:degree-comparison}
\end{table}

Table~\ref{table:error-comparison} reports the corresponding trace-estimation errors for the sinc kernel.

\begin{table}[h]
    \centering
    \begin{tabular}{c|c}
        truncated infHutch($n$) & ContHutch($\ell$) \\
        \hline
        $n = 10 \rightarrow 9.6905 \cdot 10^{-2}$ & $\ell = 0.1 \rightarrow 1.8174 \cdot 10^{-1}$ \\
         $n = 30 \rightarrow 3.1967 \cdot 10^{-2}$ & $\ell = 0.05 \rightarrow 1.0219 \cdot 10^{-1}$ \\
         $n = 100 \rightarrow 6.9704 \cdot 10^{-3}$ & $\ell = 0.025 \rightarrow 4.2606 \cdot 10^{-2}$  
    \end{tabular}
    \caption{
    Relative errors in trace estimation for the sinc kernel at $m=10^4$ samples, averaged over $100$ runs, for the same parameter values as in Table \ref{table:degree-comparison}. The rows are not intended to represent matched-accuracy parameter pairs; the table reports the accuracy obtained for the listed parameter choices.
    }
    \label{table:error-comparison}
\end{table}

\subsection{Density-of-states}
In this numerical experiment we reproduce Example 1 from \cite{zvonek2024conthutchstochastictraceestimation}. For the reader's convenience, we quickly recall a few details on the computation of operator density of states proposed in \cite{colbrook2020computingspectralmeasuresselfadjoint}.
For that purpose, operator trace estimation is combined with rational smoothing kernels. The operator we are interested in is the Schrödinger operator on a finite interval:
\begin{equation*}
    [\mathcal{L}u](x) = [-\Delta + v(x)]u(x), \quad x \in [-L,L],
\end{equation*}
where $\Delta$ is the Laplace operator, $v(x)$ is the potential function and $u(x)$ satisfies homogeneous Dirichlet on the boundary of the interval. On a finite segment $[-L,L]$ the density-of-states measure is
\begin{equation*}
    \rho_L(\lambda) = \frac{1}{2L} \sum_{k=1}^\infty \delta(\lambda - \lambda_k)
\end{equation*}
where $\lambda_1, \lambda_2, \ldots$ are the eigenvalues of $\mathcal{L}$ and $\delta(\cdot)$ is the Dirac delta measure with unit mass at the origin. As $L \to \infty$, the finite-interval density-of-states measures converge, under suitable assumptions, to the density-of-states measure of the Schr\"{o}dinger operator on the real line. In the free-particle case considered below, this limiting measure is absolutely continuous and has a piecewise smooth density $\rho$. 

As proposed in \cite{colbrook2020computingspectralmeasuresselfadjoint}, convolution with high-order rational kernels can be used to numerically approximate the density-of-states: for a $K$-th order rational kernel with prescribed simple poles $p_1,...,p_K$ and residues $r_1,...,r_K$, the convolution results in the approximation
\begin{equation}
    \label{eq:desity-of-states-convolution}
    \rho_L^{(K)}(\lambda) = \frac{1}{2L} \Tr\left[\imag\left(\sum_{j=1}^K r_j \cdot (\mathcal{L}-\lambda+p_j\sigma)^{-1}\right)\right].
\end{equation}

We compare the trace estimators by applying them to the operator inside the trace in \eqref{eq:desity-of-states-convolution}.
We consider the free particle operator ($v(x)=0$), take $L=150$ and rational kernels of orders $K=2$ and $K = 6$ with smoothing parameter $\sigma=0.2$; the kernel poles and residues are listed in \cite[Table 5.1]{colbrook2020computingspectralmeasuresselfadjoint}. The density is evaluated at $\lambda=1$. The reference values for the traces are computed directly from the known finite-interval eigenvalues, by evaluating the rational expression in \eqref{eq:desity-of-states-convolution} as a deterministic spectral sum rather than by stochastic trace estimation.
Since our goal is to assess the trace estimators, we compare the relative errors against these reference values for $\rho_L^{(K)}(\lambda)$, rather than against $\rho_L(\lambda)$ as in \cite[Figure 6.3]{zvonek2024conthutchstochastictraceestimation}.

\begin{figure}[H]
    \begin{minipage}{0.50\textwidth}
        \includegraphics[width=1\linewidth]{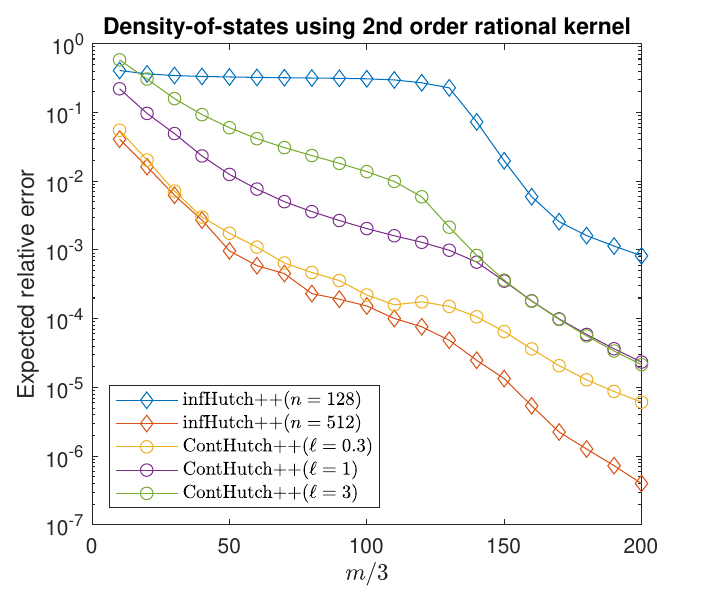}
    \end{minipage}
    \begin{minipage}{0.50\textwidth}
        \includegraphics[width=1\linewidth]{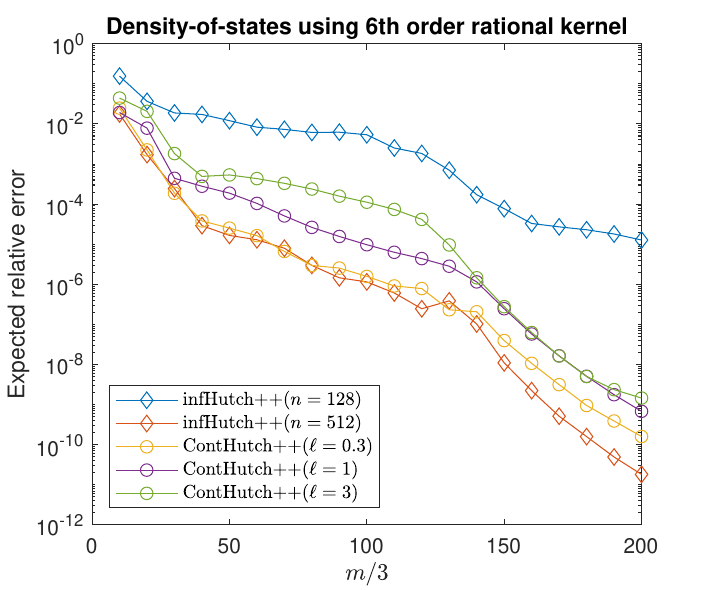}
    \end{minipage}
    \caption{Performance comparison of infHutch++ and ContHutch++ for estimating $\rho_L^{(K)}(1)$ for $K=2, 6$. On the vertical axis we plot the empirical approximation to $\EE[\frac{|\rho_L^{(K)}(1) - \hat T_m|}{\rho_L^{(K)}(1)}]$, where $\hat T_m$ denotes the trace estimate produced by the plotted method, by averaging the errors over $30$ runs for $K=2$, and over $10$ runs for $K=6$ (the reduced number of runs is due to high computational runtime of each run).}
    \label{fig:dos-comparison}
\end{figure}

Figure \ref{fig:dos-comparison} shows the comparison of the ContHutch++($\ell$) algorithm and the truncated infHutch++($n$) variant of the algorithm described in Section \ref{sec:implementation}. Once again we use Legendre polynomials for the interval $[-L, L]$ as the orthogonal basis $\{q_1, q_2, \ldots\}$. Random samples that use linear combinations of such polynomials of degree up to $512$ reach the smallest attainable relative errors for each given sample size; degree $128$ does not suffice for that purpose. As in the previous example, the procedure described in Algorithm \ref{alg:dynamic-sampling} can be used as a pilot procedure to select a finite sampling dimension before running truncated infHutch++ with fixed $n$. By adjusting the length-scale parameter $\ell$ in ContHutch++($\ell$), we can reach a similar level of accuracy (note, \cite{zvonek2024conthutchstochastictraceestimation} uses $\ell=0.3$). However, Table \ref{table:dicr-operator-size} lists the matrix sizes used internally by \chebfun in order to discretize the operators when using different sampling methods. Samples generated by the ContHutch++($\ell$) algorithm are represented by Chebyshev polynomials of far higher degree, which in turn requires finer discretization and larger linear systems, suggesting substantially higher computational cost while having similar or lower accuracy.

\begin{table}[h]
    \centering
    \begin{tabular}{c|c}
        infHutch++($n$) & ContHutch++($\ell$) \\
        \hline
        $n = 128 \rightarrow \text{size}=512$  
            & $\ell = 0.3 \to \text{size}=4096$ \\
        $ n = 512 \rightarrow \text{size}=724$ & $\ell = 0.1 \to \text{size}=4096$
    \end{tabular}
    \caption{Comparison of typical internal \chebfun matrix sizes used for operator discretization with the two operator Hutch++ variants. The size may slightly vary depending on the sample and the phase of the Hutch++ algorithm. }
    \label{table:dicr-operator-size}
\end{table}

\subsection{Quadrature-based spectral filtering} 
In the final example, we compute the spectral filter on an operator with the goal of approximating the number of its eigenvalues in a gap in its essential spectrum.
Consider the radial Dirac operator
$$
    \mathcal{D} = 
        \mb{cc}
            1 + V(r) & -\frac{d}{dr} + \frac{\kappa}{r} \\
            \frac{d}{dr} + \frac{\kappa}{r} & -1 + V(r)
        \me,
$$
where $V : [0, \infty) \to \R$ is the potential. We set the channel index $\kappa=-1$, and study short-range potentials 
 $$V\in L^\infty_0([0, \infty) )=
 \left\{W \in L^\infty([0,\infty) )~:~\lim_{R\to\infty}\|W \mathbf{1}_{\{r>R\}}\|_{L^\infty} = 0\right\}.
 $$
For the model problems in this section nothing will be lost if we assume a more stringent requirement $V\in C_\infty([0,\infty) ) \cap L^\infty_0([0, \infty) )$.

It is known that in this case the essential spectrum is $(-\infty, -1] \cup [1, \infty)$, and that the discrete spectrum lies within $(-1, 1)$ \cite[Section 1.4.3]{Thaller1992Dirac}. For a more mathematical (algebraic) treatment, one can utilize the spectral theory of block operator matrices \cite{Tretter08}. Note that the key structural feature of the potential $V\in L^\infty_0([0, \infty))$ is that it is a relatively compact potential with respect to the free (unperturbed) Dirac operator and, by Weyl's theorem on relatively compact perturbations \cite[Theorem 6.19]{teschl2014qm}, $V$ cannot perturb its essential spectrum. If $V\leq 0$ this perturbation can only move eigenvalues embedded inside the essential spectrum and depending on $\|V\|_\infty$ we will see an increased number of eigenvalues appearing in the interval $(-1, 1)$. Our goal is to estimate the number of eigenvalues in $(-1, 1)$, which we denote as $\num_V$.
For the Gaussian-well potentials considered below, this number is finite.

Counting the eigenvalues can be accomplished by computing the trace of the orthogonal projection operator
\begin{equation}
    P_{\Gamma} := \frac{1}{2\pi \ii} \int_\Gamma (z I - \mathcal{D})^{-1}\, dz,
\end{equation}
where $\Gamma$ is a positively oriented contour enclosing the isolated eigenvalues to be counted and no other part of the spectrum. Then $\Tr(P_{\Gamma})$ equals the number of eigenvalues of $\mathcal{D}$ within the contour $\Gamma$.

In the computation we first truncate the problem to the interval $[\epsilon, R]$. Let $\mathcal{D}_{\epsilon, R}$ denote the resulting self-adjoint Dirac operator with the boundary conditions specified below. The corresponding finite-interval projection is
\begin{equation}
    \label{eq:projector}
    P_{\epsilon, R, \Gamma}
    = \frac{1}{2\pi \ii} \int_\Gamma (z I - \mathcal{D}_{\epsilon, R})^{-1}\, dz .
\end{equation}
Since $\mathcal{D}_{\epsilon, R}$ has discrete spectrum, this projection is finite rank whenever $\Gamma$ encloses finitely many eigenvalues of $\mathcal{D}_{\epsilon, R}$, and $\Tr(P_{\epsilon, R, \Gamma})$ equals the number of finite-interval eigenvalues enclosed by $\Gamma$. To approximate the integral in \eqref{eq:projector} we use a $Q$-point quadrature on the contour. If we denote the quadrature nodes and weights respectively as $z_i$ and $w_i$, $i=1, \ldots, Q$, this induces the operator
$$
    \A_{\epsilon, R, Q}        
        = \rho(\mathcal{D}_{\epsilon, R})
        := \sum_{j=1}^Q w_j (z_j I - \mathcal{D}_{\epsilon, R})^{-1}.
$$
Here the rational function $\rho$ of order $Q$ with poles $z_1, \ldots, z_Q$ is called a filter; an ideal filter for $P_{\epsilon,R,\Gamma}$ is the indicator function of the part of the real line enclosed by $\Gamma$. This technique of eigenvalue counting is well-known in the matrix case \cite{DiNapoliPolizziSaad16}; contour-filter constructions are the basis of filtered subspace iteration for self-adjoint operators as well \cite{GopalakrishnanGrubisicOvall20}. We therefore run the stochastic trace estimation algorithms with the goal of computing $\Tr(\A_{\epsilon, R, Q}) = \Tr(\rho(\mathcal{D}_{\epsilon, R}))$. The relation between the reported trace and the desired count is therefore
$$
    \Tr(\A_{\epsilon, R, Q})
        \approx \Tr(P_{\epsilon, R, \Gamma})
        \approx \Tr(P_\Gamma),
$$
where the first approximation is the quadrature/filtering error and the second is the finite-interval truncation error.

In order to apply stochastic trace estimation to $\A_{\epsilon, R, Q}$, we have to ensure that it is a positive trace-class operator. In numerical experiments below, we use the shifted trapezoidal rule on an ellipse. Let $Q$ be even. We write the ellipse in the form
$$
    z(\theta)
        = c + \frac{\eta}{2}
            \left(
                \tau e^{\ii \theta} + \tau^{-1}e^{-\ii \theta}
            \right),
    \quad \tau>1,
$$
where $c \in \R$ is the center and $\eta>0$. The real half-width
of the ellipse is $\gamma = \frac{\eta}{2}(\tau + \tau^{-1})$.
We use the shifted trapezoidal nodes
$$
    \theta_j = \frac{2\pi}{Q} \left( j + \frac{1}{2} \right),
    \quad 
    j = 1, \ldots, Q,
$$
and set the quadrature nodes and weights as
$$
    z_j = z(\theta_j),
    \quad
    w_j = \frac{z'(\theta_j)}{\ii Q}.
$$
The corresponding rational filter is
$
    \rho(x)
        = \sum_{j=1}^{Q} \frac{w_j}{z_j - x}.
$
The relation between trapezoidal quadrature on an ellipse and Chebyshev filters is standard in the FEAST literature; see
\cite[Lemma~3.1]{GuettelPolizziTangViaud15}. With the notation above, the filter can be written, for $x \in \R$, as
$$
    \rho(x)
        = \frac{\tau^Q - \tau^{-Q}}
            { \tau^Q+\tau^{-Q}
                + 2 T_Q \left( \frac{x - c}{\eta} \right) },
$$
where $T_Q$ is the Chebyshev polynomial of the first kind. Since $Q$ is even, $T_Q(t) \geq -1$ for every $t \in \R$. Hence
$$
    \tau^Q+\tau^{-Q} + 2T_Q(t)
        \geq \tau^Q+\tau^{-Q}-2
        > 0,
$$
and the numerator $\tau^Q-\tau^{-Q}$ is positive. Consequently,
$$
    \rho(x)>0, \quad \text{for all } x \in \R.
$$
By the spectral theorem, $\A_{\epsilon, R, Q} = \rho(\mathcal{D}_{\epsilon, R})$ is positive semidefinite.
Moreover, since $T_Q$ is a polynomial of degree $Q$, the same
representation gives
$$
    \rho(x) = \mathcal{O} (|x|^{-Q}), \quad |x| \to \infty.
$$
With the boundary conditions specified below, $\mathcal{D}_{\epsilon, R}$ is a regular first-order self-adjoint system on a bounded interval. Thus it has compact resolvent and its eigenvalues $\lambda_j(\mathcal{D}_{\epsilon, R})$ are growing linearly in magnitude. Therefore
$$
    \sum_j \rho(\lambda_j(\mathcal{D}_{\epsilon, R})) < \infty
$$
for $Q>1$. Hence $\A_{\epsilon, R, Q} = \rho(\mathcal{D}_{\epsilon, R})$ is a positive trace-class operator, and we can apply the stochastic trace estimation algorithms of Section \ref{sec:optrace}.

The exact projection $P_{\epsilon, R, \Gamma}$ is finite rank, while $\A_{\epsilon, R, Q}$ is generally not. Instead, $\A_{\epsilon, R, Q}$ is a positive trace-class rational approximation to this projection. Its trace has the spectral representation
$$
    \Tr(\A_{\epsilon, R, Q})
    = \sum_j \rho( \lambda_j( \mathcal{D}_{\epsilon, R} ) ).
$$
Eigenvalues of $\mathcal{D}_{\epsilon, R}$ well inside the contour contribute weights $\rho(\lambda_j)$ close to one, while eigenvalues well outside the contour contribute weights close to zero. Eigenvalues close to the contour may be only partially counted. Thus $\Tr(\A_{\epsilon, R, Q})$ is a smoothed quadrature-based approximation to the finite-interval contour count $\Tr(P_{\epsilon, R, \Gamma})$.

To demonstrate this in numerical experiments, we consider potentials consisting of several Gaussian wells, i.e., potentials of the form
$$
    V(r) = \lambda \cdot \sum_{i=1}^k e^{-(r - r_i)^2 / a^2}.
$$
This model describes elastic Dirac fermion scattering and the effective potential $V$ is determined from the density of the nuclear charge, which can be modeled as the sum of the Gaussian terms \cite{Valuev2020Skyrme}. We control the number of eigenvalues in $(-1, 1)$ by varying the coupling parameter $\lambda<0$ and the number of wells $k$.

For quadrature we use the shifted trapezoidal rule with $Q=64$ nodes on an ellipse with real-axis endpoints $-0.99$ and $0.99$, and vertical semiaxis $0.1$. This quadrature results in the filter $\rho$ taking the following values:
\begin{center}
    \begin{tabular}{c||c|c|c|c|c|c|c}
        $|x|$ & 0.9 & 0.98 & 0.99 & 1 & 1.01 & 1.05 & 1.1 \\ \hline
        $\rho(x)$ & 1.00e+00 & 9.96e-01 & 5.00e-01 & 9.08e-03 & 3.62e-04 & 5.87e-08 & 3.23e-11
    \end{tabular}
\end{center}

\noindent
Eigenvalues well inside $[-0.99, 0.99]$ therefore contribute approximately one to $\Tr(\rho(\mathcal{D}_{\epsilon, R}))$, while eigenvalues close to $\pm 1$ may be only partially counted or effectively ignored by the filter.

All random samples in this example are generated in the Hilbert space $L^2([\epsilon,R])^2$, with inner product
$
\langle u, v\rangle = \int_{\epsilon}^R u_1(r) \overline{v_1(r)} + u_2(r) \overline{v_2(r)} \,dr .
$
Note that the random samples are real-valued but complex arithmetic enters through the shifted resolvent solves.
The truncated endpoint $R$ is chosen separately for each potential and is reported below; $\epsilon$ was set to $10^{-2}$. We consider two different bases for infHutch++. First, we use finite element method (FEM) to solve shifted systems with the Dirac operator. To generate random samples, we take linear combinations of the orthonormalized FEM basis. More precisely, let $\varphi_1, \ldots, \varphi_{N}$ denote the FEM basis for the ``top'' function in the Dirac operator (in our case, these will be the P2 elements basis), and let $\psi_1, \ldots, \psi_N$ denote the FEM basis for the ``bottom'' function in the Dirac operator (in our case, these will be $\psi_j = \frac{1}{2}(\varphi'_j + \frac{\kappa}{r} \varphi_j)$, i.e., we employ kinetic balance). To get the orthonormal basis, we build the mass matrix (block diagonal matrix with blocks $[\langle \varphi_i, \varphi_j \rangle]_{i,j=1, \dots, N}$ and $[\langle \psi_i, \psi_j \rangle]_{i,j=1, \dots, N}$) and compute its Cholesky factor $\hat{R}$. Then the orthogonal basis consists of columns of the matrix
$$
    \mb{cccccc}
        \hat{\varphi}_1 & \ldots & \hat{\varphi}_N & 0 & \ldots & 0 \\
        0 & \ldots & 0 & \hat{\psi}_1 & \ldots & \hat{\psi}_N
    \me
    = 
    \mb{cccccc}
        \varphi_1 & \ldots & \varphi_N & 0 & \ldots & 0 \\
        0 & \ldots & 0 & \psi_1 & \ldots & \psi_N
    \me
    \hat{R}^{-1}.
$$
We then generate the random samples for infHutch++ as
$$
    x = 
        \mb{c}
            \sum_{i=1}^N \alpha_i \hat{\varphi}_i \\
            \sum_{i=1}^N \beta_i \hat{\psi}_i 
        \me,
    \quad
        \alpha_i, \beta_i \sim N(0, 1).
$$
Using $n$ nodes in $[\epsilon, R]$ and P2 elements, we have $N=2n-2$, and in total $4n-4$ basis functions once the boundary conditions are imposed.

The second way of generating random samples is to consider scaled Legendre polynomials $L_0, \ldots, L_n$ so that they are orthonormal on the interval $[\epsilon, R]$, and draw
$$
    x = 
        \mb{c}
            \sum_{i=0}^n \alpha_i L_i \\
            \sum_{i=0}^n \beta_i L_i 
        \me,
    \quad
        \alpha_i, \beta_i \sim N(0, 1).
$$
Then we simply use \chebfun to define the Dirac operator and the random samples, and use the backslash operator to solve with the resolvents as in the previous example.
For both the FEM and Legendre bases, homogeneous Dirichlet conditions were imposed for the ``bottom'' function at $r=\epsilon$, and for the ``top'' function at $r=R$ in all solves.

\begin{figure}[t]
    \begin{minipage}{0.50\textwidth}
        \includegraphics[width=1\linewidth]{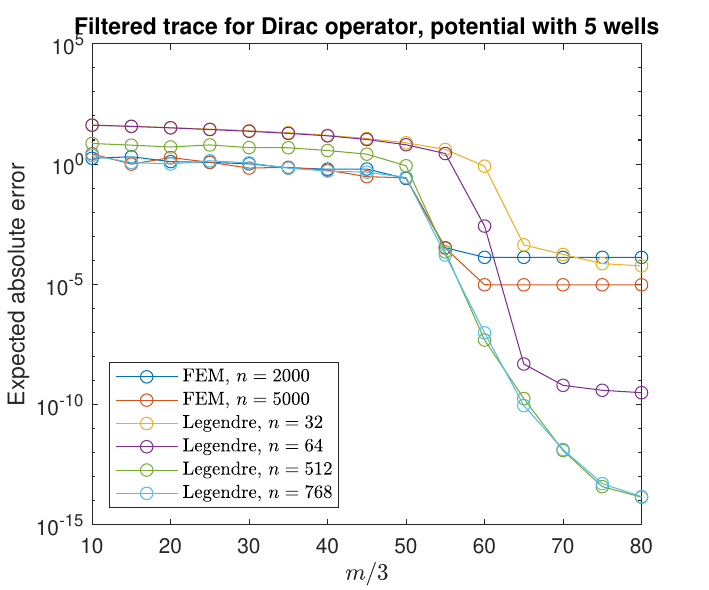}
    \end{minipage}
    \hfill
    \begin{minipage}{0.50\textwidth}
        \includegraphics[width=1\linewidth]{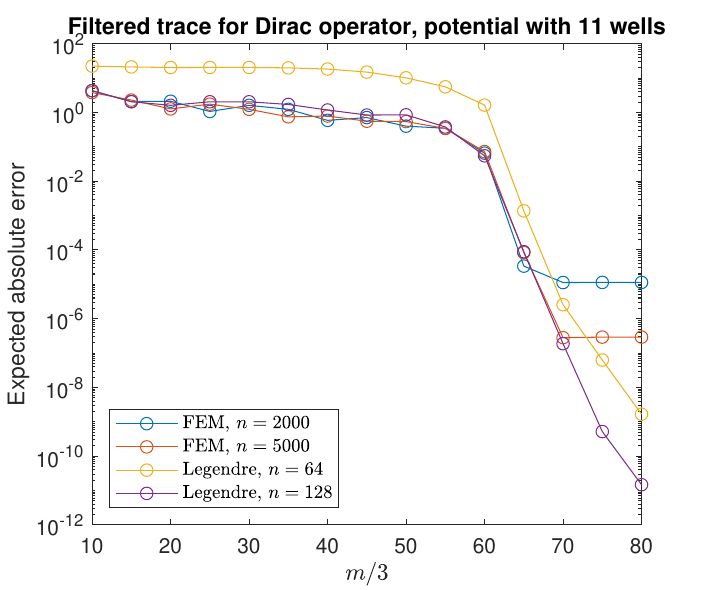}
    \end{minipage}
    \caption{Stochastic trace estimation of the filtered operator $\rho(\mathcal{D}_{\epsilon, R})$, using the truncated infHutch++ algorithm. Random samples are generated as linear combinations of orthonormalized FEM basis functions, and as linear combinations of Legendre polynomials. Averaged over 10 runs.}
    \label{fig:dirac-count}
\end{figure}

The results are shown in Figure \ref{fig:dirac-count}. 
For the left plot, we take the potential $V$ with $\lam=-10$, $a=5$, and $k=5$ wells centered at $r_i = 15i$. There are $\num_V = 52$ eigenvalues in the gap $(-1, 1)$; the smallest is $\lam_1 \approx -0.96773$, and the largest $\lam_{52} \approx 0.97845$. All gap eigenvalues of the Dirac operator\footnote{The reported gap eigenvalues were computed independently using $\texttt{eigs}$ in \chebfun and also the same kinetic-balance FEM discretization, with convergence checked under mesh refinement and enlargement of the interval $[\epsilon,R]$. Eigenvalues were reported only if they remained stable under these refinements.} lie well inside $[-0.99, 0.99]$.
Next we fix $\epsilon=10^{-2}$ and $R=120$ to truncate the domain, and use $\Tr(\rho(\mathcal{D}_{\epsilon, R})) \approx 52.006128204243$ as the reference value: using \chebfun we computed all eigenvalues of $\mathcal{D}_{\eps, R}$ in a wider interval $[-1.2, 1.2]$, applied the filter function $\rho$ to each of the eigenvalues, and summed all the obtained values\footnote{Enlarging the eigenvalue interval did not change the displayed digits of the reference value.}.
For FEM discretization we use $n=2000$ and $n=5000$ nodes, respectively, meaning that the random samples are linear combinations of $7996$ and $19996$ basis functions, respectively. We consider Legendre polynomials with degrees up to $n=32$, $n=64$, $n=512$, $n=768$, respectively, resulting in truncated infHutch++ bases of sizes $66$, $130$, $1026$, $1538$, respectively. For each of these bases, the infHutch++($n$) algorithm using $m/3=80$ samples always reports the same value in $52.0061 \pm 0.0002$ (the exact value depends on the sampling method). 
Note that both FEM bases and the Legendre basis with $n=768$ require only $m/3=10$ samples to approximate the filtered trace $\Tr(\rho(\mathcal{D}_{\epsilon, R}))$, which is close to the true gap count $\num_V=52$, with an average error less than $2$. The necessity of using a large polynomial degree in order to get more accurate approximations with smaller number of samples
may be related to the highly oscillatory nature of the relevant eigenfunctions.

In the right plot, we consider the potential with $\lam=-1$, $a=5$, and $k=11$ wells centered at $r_i = 15i-15$. There are $\num_V = 61$ eigenvalues in $(-1, 1)$; the smallest is $\lam_1 \approx 0.13484$, and the largest two $\lam_{60} \approx 0.98315$, $\lam_{61} \approx 0.99593$. The ellipse used here has real-axis endpoints $\pm 0.99$, so $\lambda_{61}$ lies outside the contour. Thus the contour count is $60$, although the full gap count is $61$; in turn, the reference value is
$\Tr(\rho(\mathcal{D}_{\epsilon, R})) \approx 60.024852949199$ (computed in the same way as for the left plot). Counting the last
eigenvalue would require choosing a contour closer to the spectral edge at $1$; resolving such a narrow transition region may require a significantly larger number of quadrature nodes. On the other hand, the eigenfunctions in this case are less oscillatory, and we only need $m/3=15$ random samples using Legendre polynomials of degree $n=128$ to approximate the filtered trace with an average error less than $2$. With $80$ samples, all four methods shown in the figure reported the value $60.02485 \pm 0.00001$. The domain was truncated with $\epsilon = 10^{-2}$ and $R=200$ in this example.

\section{Conclusion}
\label{sec:conclusion}

In this paper, we have introduced infinite-dimensional analogues of the Girard--Hutchinson and Hutch++ trace estimators for positive trace-class operators on separable Hilbert spaces. The idealized estimators use Gaussian random elements whose covariance is determined by the target operator.
This leads to unbiased operator versions of the Girard--Hutchinson and
Hutch++ estimators.
For these idealized algorithms we prove high-probability error bounds that
are direct infinite-dimensional counterparts of the finite-dimensional case;
in particular, infHutch++ has sample complexity
$\mathcal{O}(\eps^{-1})$.
In contrast to the externally prescribed squared-exponential covariance
used in ContHutch++, our covariance operators are determined by the target
operator itself. This avoids the associated smoothing bias in the idealized
estimators and provides a more direct analogue of the matrix setting.

Since the idealized sampling distributions are generally not available in practice, we have also propose truncated implementations based on sampling from finite-dimensional subspaces. For the Girard--Hutchinson estimator, this truncation can be interpreted as applying the matrix estimator to a finite-rank compression of the operator.
For infHutch++, the practical range finder restricts the randomness to a
finite-dimensional sampling space while keeping the images under the operator
in the ambient Hilbert space. For fixed sample budget, we have shown that the
resulting estimator converges in distribution to idealized infHutch++ as the
basis dimension tends to infinity, and that its expectation converges to
$\Tr(\A)$.

The numerical experiments illustrate practical advantages of this sampling
strategy in the \chebfun implementation.
In the integral-operator examples, truncated infHutch using Legendre bases reaches accuracy comparable to the idealized benchmark while requiring significantly lower-degree Chebyshev representations than the squared-exponential Gaussian-process samples used by ContHutch. In the density-of-states experiment, truncated infHutch++ achieves similar or better stochastic accuracy than ContHutch++ with substantially smaller internal \chebfun discretizations. The final experiment illustrates the use of truncated
infHutch++ for quadrature-based spectral filtering of a differential operator, with the
trace serving as an approximation to an eigenvalue count.

\section{Funding}
The work of ZB, LG, HO was supported by the following projects: Croatian Science Foundation grants IP-2025-02-3733 (``Data driven identification and dimension reduction of dynamical systems'') and IP-2022-10-5191 (``Optimal control and model reduction for evolution and data driven problems''), the project ``Implementation of cutting-edge research and its application as part of the Scientific Center of Excellence for Quantum and Complex Systems, and Representations of Lie Algebras'', Grant No.~PK.1.1.10.0004, co-financed by the European Union through the European Regional Development Fund - Competitiveness and Cohesion Programme 2021--2027, and by the European Union - NextGenerationEU through the institutional project Impact4MATH at the University of Zagreb, Faculty of Science.

\bibliographystyle{plain}
\bibliography{refs.bib}

\end{document}